\edef\savecatcodeat{\the\catcode`@}
\catcode`\@=11

\def\tb@ifSpecChars#1#2{#1}
\def\tb@ifNoSpecChars#1#2{#2}

\def\tableau{%
  \bgroup
  \@ifstar{\let\Tif\tb@ifNoSpecChars\tb@tableauB}
          {\let\Tif\tb@ifSpecChars\tb@tableauB}}

\def\tb@tableauB{
  \@ifnextchar[{\tb@tableauC}{\tb@tableauC[]}}

\def\tb@tableauC[#1]{\hbox\bgroup%
    \let\\=\cr
    \def\bl{\global\let\tbcellF\tb@cellNF}%
    \def\tf{\global\let\tbcellF\tb@cellH}
    \dimen2=\ht\strutbox \advance\dimen2 by\dp\strutbox%
    \ifx\baselinestretch\undefined\relax%
    \else%
       \dimen0=100sp \dimen0=\baselinestretch\dimen0%
       \dimen2=100\dimen2 \divide\dimen2 by\dimen0%
    \fi%
    \let\tpos\tb@vcenter
    \tb@initYoung
    \tb@options#1\eoo
    \let\arrow\tb@arrow%
    \dimen0=\Tscale\dimen2%
    \dimen1=\dimen0 \advance\dimen1 by \tb@fframe%
    \lineskip=0pt\baselineskip=0pt
    \def\tb@nothing{}%
    \def\endcellno{$\rss\egroup\bss\egroup}
    \def\endcell{\endcellno\kern-\dimen0}
    \def\begincell{\vbox to\dimen0\bgroup\vss\hbox to\dimen0\bgroup\hss$}%
    \let\overlay\tb@overlay%
    \let\fl\tb@fl%
    \let\lss\hss\let\rss\hss\let\tss\vss\let\bss\vss
    \def\mkcell##1{
        \let\tbcellF\tb@cellD
        \def\tb@cellarg{##1}
        \ifx\tb@cellarg\tb@nothing\let\tb@cellarg\tb@cellE\fi%
        \begincell\tb@cellarg\endcellno
        \tbcellF}
    \let\savecellF\tbcellF
     \Tif{\catcode`,=4\catcode`|=\active}{}\tb@tableauD}%

\let\tb@savetableauD\tableauD
{
    \catcode`|=\active \catcode`*=\active \catcode`~=\active%
    \catcode`@=\active
\gdef\tableauD#1{%
  \Tif{
    \mathcode`|="8000 \mathcode`*="8000%
    \mathcode`~="8000 \mathcode`@="8000%
    \def@{\bullet}%
    \let|\cr
    \let*\tf
    \let~\sk
  }{}%
  \tpos{\tabskip=0pt\halign{&\mkcell{##}\cr#1\crcr}}%
  \global\let\tbcellF\savecellF
  \egroup
  \egroup}
}
\let\tb@tableauD\tableauD
\let\tableauD\tb@savetableauD
\let\tb@savetableauD\undefined

\def\tb@options#1{\ifx#1\eoo\relax\else\tb@option#1\expandafter\tb@options\fi}

\def\tb@option#1{%
  \if#1t\let\tpos\tb@vtop\fi
  \if#1c\let\tpos\tb@vcenter\fi
  \if#1b\let\tpos\vbox\fi
  \if#1F\tb@initFerrers\fi
  \if#1Y\tb@initYoung\fi
  \if#1s\tb@initSmall\fi
  \if#1m\tb@initMedium\fi
  \if#1l\tb@initLarge\fi
  \if#1p\tb@initPartition\fi
  \if#1a\tb@initArrow\fi
}

\def\tb@vcenter#1{\ifmmode\vcenter{#1}\else$\vcenter{#1}$\fi}

\def\tb@vtop#1{\hbox{\raise\ht\strutbox\hbox{\lower\dimen0\vtop{#1}}}}

\def\tb@initPartition{\def\Tscale{.3}}
\def\tb@initSmall{\def\Tscale{1}}
\def\tb@initMedium{\def\Tscale{2}}
\def\tb@initLarge{\def\Tscale{3}}

\def\tb@initArrow{\dimen2=1.25em}

\def\tb@initYoung{%
  \def\tb@cellE{}
  \let\tb@cellD\tb@cellN
  \def\sk{\global\let\tbcellF\tb@cellNF}}
\def\tb@initFerrers{%
  \def\tb@cellE{\bullet}
  \let\tb@cellD\tb@cellNF
  \def\sk{\bullet}}

\tb@initMedium

\def\tb@sframe#1{%
  \vbox to0pt{
    \vss
    \hbox to0pt{%
      \hss
      \vbox to\dimen1{
        \hrule depth #1 height0pt
        \vss
        \hbox to\dimen1{
          \vrule width #1 height\dimen1
          \hss
          \vrule width #1
          }%
        \vss
        \hrule height #1 depth 0in
        }%
      \kern-\tb@hframe
      }%
    \kern-\tb@hframe}}

\def\tb@hframe{.2pt}\def\tb@fframe{.4pt}\def\tb@bframe{2pt}
\def\tb@cellH{\tb@sframe{\tb@bframe}}       
\def\tb@cellNF{}                            
\def\tb@cellN{\tb@sframe{\tb@fframe}}       
\let\tbcellF\tb@cellN                       

\def\tb@rpad{1pt}
\def\tb@lpad{1pt}
\def\tb@tpad{1.8pt}
\def\tb@bpad{1.8pt}

\def\tb@overlay{\endcell\@ifnextchar[{\tb@overlaya}{\begincell}}
\def\tb@overlaya[#1]{\vbox to\dimen0\bgroup%
  \tb@overlayoptions#1\eoo%
  \tss\hbox to\dimen0\bgroup\lss$}
\def\tb@overlayoptions#1{\ifx#1\eoo\relax\else\tb@overlayoption#1\expandafter\tb@overlayoptions\fi}

\def\tb@overlayoption#1{
  \if#1t\def\tss{\vskip\tb@tpad}\let\bss\vss\fi
  \if#1c\let\tss\vss\let\bss\vss\fi
  \if#1b\def\bss{\vskip\tb@bpad}\let\tss\vss\fi
  \if#1l\def\lss{\hskip\tb@lpad}\let\rss\hss\fi
  \if#1m\let\lss\hss\let\rss\hss\fi
  \if#1r\def\rss{\hskip\tb@rpad}\let\lss\hss\fi
}

\def\tb@fl{\endcell\begincell\vrule depth 0pt width \dimen0 height \dimen0 \endcell\begincell}

\@ifundefined{diagram}{}{
\def\tb@arrowpad{.5}

\newoptcommand{\tb@arrow}{\@ne}[2]{%
  \endcell
   \begingroup%
   \let\dg@getnodesize\tb@getnodesize
   \dg@USERSIZE=#1\relax%
   \ifnum\dg@USERSIZE<\@ne \dg@USERSIZE=\@ne \fi%
   \dg@parse{#2}%
   \dg@label{\tb@draw{#1}{#2}}}

\def\tb@getnodesize#1#2#3#4#5{\dimen3=\tb@arrowpad\dimen2 #4=\dimen3 #5=\dimen3\relax}
\def\tb@getnodesize#1#2#3#4#5{\ifnum#2=0\ifnum#3=0\tb@getnodesizetail{#4}{#5}\else\tb@getnodesizehead{#4}{#5}\fi\else\tb@getnodesizehead{#4}{#5}\fi}
\def\tb@getnodesizetail#1#2{\dimen3=.5\dimen2 #1=\dimen3 #2=\dimen3}
\def\tb@getnodesizehead#1#2{\dimen3=.5\dimen2 #1=\dimen3 #2=\dimen3}

\def\tb@draw#1#2#3#4{%
        \dg@X=0\dg@Y=0\dg@XGRID=1\dg@YGRID=1\unitlength=.001\dimen0%
        \dg@LBLOFF=\dgLABELOFFSET \divide\dg@LBLOFF\unitlength%
        \dg@drawcalc
        \begincell
        \let\lams@arrow\tb@lams@arrow
        \begin{picture}(0,0)\begingroup\dg@draw{#1}{#2}{#3}{#4}\end{picture}%
        \endcell
        \endgroup
        \begincell}
}

\def\tb@lams@arrow#1#2{%
 \lams@firstx\z@\lams@firsty\z@
 \lams@lastx#1\relax\lams@lasty#2\relax
 \lams@center\z@
 \N@false\E@false\H@false\V@false
 \ifdim\lams@lastx>\z@\E@true\fi
 \ifdim\lams@lastx=\z@\V@true\fi
 \ifdim\lams@lasty>\z@\N@true\fi
 \ifdim\lams@lasty=\z@\H@true\fi
 \NESW@false
 \ifN@\ifE@\NESW@true\fi\else\ifE@\else\NESW@true\fi\fi
 \ifH@\else\ifV@\else
  \lams@slope
  \ifnum\lams@tani>\lams@tanii
   \lams@ht\ten@\p@\lams@wd\ten@\p@
   \multiply\lams@wd\lams@tanii\divide\lams@wd\lams@tani
  \else
   \lams@wd\ten@\p@\lams@ht\ten@\p@
   \divide\lams@ht\lams@tanii\multiply\lams@ht\lams@tani
  \fi
 \fi\fi
 \ifH@  \lams@harrow
 \else\ifV@ \lams@varrow
 \else \lams@darrow
 \fi\fi
}

\catcode`\@=\savecatcodeat
\let\savecatcodeat\undefined

\documentclass[tbtags,reqno]{amsart}
\usepackage{graphicx,color}
\usepackage{amsmath,amsfonts,amstext,amsbsy,amsopn,amsthm,amscd,amsxtra,dsfont,amssymb,pb-diagram}
\usepackage{amsmath,amsthm,amsopn,amsfonts,amssymb,epsfig,enumerate}
\usepackage{amsthm,amsopn,amsfonts,amssymb,epsfig,mathrsfs,cancel,wasysym}
\usepackage[all]{xy}
\usepackage[active]{srcltx}

\numberwithin{equation}{section}

\makeatletter
\renewcommand{\subsubsection}{\@startsection
{subsubsection} {3} {0mm} {\baselineskip} {-0.5\baselineskip} {\normalfont\normalsize\bfseries}} \makeatother

\newtheorem{theorem}{Theorem}
\newtheorem{lemma}[theorem]{Lemma}
\newtheorem{proposition}[theorem]{Proposition}

\newtheorem{conjecture}[theorem]{Conjecture}

\newtheorem{remark}[theorem]{Remark}
\newtheorem*{acknow}{Acknowledgments}

\def\la{{\lambda}}

\def\cal L{{\mathcal L}}

\newcommand{\La}{\Lambda}

\def \part {\vdash}
\def\La {\Lambda}

\def \Om {\Omega}
\def \ta {\theta}
\def\cd{\circledast}

\newcommand{\tcercle}[1]{\ensuremath{\setlength{\unitlength}{1ex}\begin{picture}(2.8,2.8)\put(1.4,1.4){\circle{2.8}\makebox(-5.6,0){#1}}\end{picture}}}

\newcommand{\cerclep}{\ensuremath{\setlength{\unitlength}{1ex}\begin{picture}(2.8,2.8)\put(1.4,1.4){\circle*{2.7}\makebox(-5.6,0)}\end{picture}}}

\begin{document}

\title[Pieri rules for the Jack polynomials in superspace]
{Pieri rules for the Jack polynomials in superspace and the 6-vertex model}

\author{Jessica Gatica }
\address{Facultad de Matem\'aticas, Pontificia Universidad Cat\'olica de Chile, Casilla 306, Correo 22, Santiago, Chile} \email{jegatiga@mat.uc.cl}

\author{Miles Jones }
\address{Department of Computer Science, University of California, San Diego, 9500 Gilman Drive 0404, La Jolla, CA 92093, USA}
\email{mej016@ucsd.edu}

\author{Luc Lapointe}
\address{Instituto de Matem\'atica y F\'{\i}sica, Universidad de Talca, Casilla 747, Talca,
Chile} \email{lapointe@inst-mat.utalca.cl}


\begin{abstract}  We present Pieri rules for the Jack polynomials in superspace.  
 The coefficients in the Pieri rules are,  except for an extra determinant,
 products of quotients of linear factors in $\alpha$ (expressed,   
 as in the usual Jack polynomial case, in terms of certain hook-lengths in a Ferrers' diagram).  We show that, surprisingly, the extra determinant is related to the partition function of the 6-vertex model.  We give, as a conjecture, the Pieri rules for the Macdonald polynomials in superspace.
\end{abstract}


\maketitle

\section{Introduction}

The Jack polynomials in superspace, $P_\Lambda^{(\alpha)}$, have been introduced  in connection with the 
supersymmetric version of the  Calogero-Sutherland model \cite{FM}.
They are indexed by superpartitions (see Section~\ref{secdef} for the relevant definitions) and depend on the anticommuting
variables $\theta_1,\theta_2,\dots$ as well as the usual variables $z_1,z_2,\dots$.
Most of the
fundamental properties of the Jack polynomials in superspace  have already been established \cite{DLM2,DLM3}.  These
properties, such as the norm, specialization or duality, generalize beautifully those of the usual Jack polynomials
in superspace.
A notable exception up until now was that of the Pieri rules, for which there was not even a conjectured formula.\footnote{We should mention that special cases were considered in \cite{Br}.}  To be more precise,
let the Pieri coefficients for the Jack polynomials in superspace be defined by
(we refer to Section~\ref{secdef} for the relevant definitions)
\begin{equation} \label{pierifirst}
e_n \, P_\Lambda^{(\alpha)} = \sum_\Omega v_{\Lambda \Omega}(\alpha) \, P_\Omega^{(\alpha)} 
\qquad {\rm and} \qquad \tilde e_n \, P_\Lambda^{(\alpha)} = \sum_\Omega \tilde v_{\Lambda \Omega}(\alpha) \, P_\Omega^{(\alpha)} 
\end{equation}
The Pieri coefficients
$v_{\Lambda \Omega}(\alpha)$  and $\tilde v_{\Lambda \Omega}(\alpha)$ happen to be much more complicated than in the usual Jack polynomials case due to the presence in some cases of a non-linear factor.  For instance, given the superpartitions
\begin{equation} \label{spartsex} 
{\Lambda = \tiny{{{\tableau*[scY]{  & & &&&&\bl  \bigcirc    \cr & &  &&
 \cr  &&&&\bl  \bigcirc  \cr  &&&\bl\bigcirc\cr&\cr \cr }}}}
\qquad {\rm and }\qquad
\Omega = \tiny{{{\tableau*[scY]{  & & &&&&   \cr & &  && & \bl  \bigcirc 
 \cr  &&&& \cr  &&&  \cr&&\bl  \bigcirc \cr \cr\bl  \bigcirc  }}}}  }
\end{equation}
the coefficient
of $P_\Omega^{(\alpha)}$ in the product
$e_3 P_\Lambda^{(\alpha)}$ is given by 
\begin{equation} \label{ugly}
  {\frac{1}{1152}\frac{\alpha^4(2\alpha+3)(3\alpha+4){(416\alpha^6+2000\alpha^5+3484\alpha^4+2608\alpha^3+559\alpha^2-256\alpha-108)}}
    {(4\alpha+3)(5\alpha+4)(7\alpha+6)(2\alpha+1)(\alpha+1)^{10}}}
\end{equation}
Understanding the non-linear factor thus appears at first glance to be a daunting task.

It was first determined in \cite{DLM3} that
$v_{\Lambda \Omega}(\alpha)$ (resp. $\tilde v_{\Lambda \Omega}(\alpha)$) is not equal to zero only if
$\Omega/\Lambda$ is a vertical $n$-strip (resp. vertical $\tilde n$-strip).  The main goal of this paper  is to
obtain genuine Pieri rules for the Jack polynomials in superspace, that is, 
to provide explicit formulas for the  coefficients,
$v_{\Lambda \Omega}(\alpha)$  and $\tilde v_{\Lambda \Omega}(\alpha)$ (see Theorem~\ref{theo}).
Remarkably, the non-linear factor described previously turns out
to be a determinant related to the partition function of the 6-vertex model in statistical mechanics while the remaining linear factors can be expressed,   
as in the usual Jack polynomial case, in terms of certain hook-lengths in
a Ferrers' diagram.

The proof of the Pieri rules will be done by induction using the relations $\{\tilde e_{n-1},q^\perp\} = n \, e_{n}$ and
$\tilde e_0 e_{n} -[Q,e_{n}]= \tilde e_{n}$, where  $q^\perp$ and $Q$ are operators that belong the negative-half of the super Virasoro algebra.
The inductive process is possible because the explicit action of $\tilde e_0, q^\perp$ and $Q$
on Jack polynomials in superspace was obtained in \cite{DLM4}.
Dual Pieri rules involving $\alpha$-deformations of the homogeneous symmetric functions in superspace then follow immediately from a duality property of the Jack polynomials in superspace (see Theorem~\ref{theodual}).

The Pieri rules for the Jack polynomials in superspace appear to have a natural extension to the Macdonald polynomials in superspace
(see Conjecture~\ref{conj}).
The connection with the partition function of the 6-vertex model is even more natural in the
Macdonald case since the non-linear factor becomes the Izergin-Korepin determinant after a simple change of variables
(to obtain the Jack case, we have to then do the limit $q=t^\alpha$, $t\to 1$). 
We should mention that in the non-supersymmetric case a different connection between Macdonald polynomials and the partition function of the 6-vertex model was observed in \cite{Wa} when studying a bisymmetric function 
introduced in \cite{KN} and appearing in the action of the Macdonald operators acting on the Cauchy kernel.  We are hopeful that this new connection will provide further insight into the combinatorics of alternating sign matrices.

Here is the outline of the article.  After giving in Section~\ref{secdef} the necessary background on symmetric function theory in superspace,
we present the Pieri rules for the Jack polynomials in superspace in Section~\ref{SecPieri}, as well as the main idea of the proof.  The 
technical details of the proof are then provided in Section~\ref{secproofs} which forms the bulk of the article.
The conjectured Pieri rules for the Macdonald polynomials in superspace are given in Section~\ref{secMac} while  the connection with the Izergin-Korepin determinant related to the partition functions of the
6-vertex model is made in Section~\ref{sec6vertex}.  Finally, we give explicitly the dual Pieri rules involving $\alpha$-deformations of the homogeneous symmetric functions in superspace in Appendix~\ref{SecPieridual}.

\section{Definitions} \label{secdef}

A polynomial {in} superspace, or equivalently, a superpolynomial, is
a polynomial in the usual $N$ variables $z_1,\ldots ,z_N$  and the $N$ anticommuting variables $\ta_1,\ldots,\ta_N$ over a certain field, which will be taken in the remainder of this article to be $\mathbb Q$.  
A superpolynomial $P(z,\theta)$,
with $z=(z_1,\ldots,z_N)$ and $\theta=(\theta_1,\ldots,\theta_N)$, is said to be symmetric if the following is satisfied:
\begin{equation}
P(z_1,\dots,z_N,\theta_1,\dots,\theta_N) = P(z_{\sigma(1)},\dots,z_{\sigma(N)},\theta_{\sigma(1)},\dots,\theta_{\sigma(N)}) \qquad 
\forall \, \sigma \in S_N
\end{equation}
where $S_N$ is the symmetric group on $\{1,\dots, N \}$.  The ring of superpolynomials in $N$ variables has a natural grading with respect to the fermionic 
degree $m$ (the total degree in the anticommuting variables).
We will denote
by $\Lambda^m_N$ the ring of symmetric superpolynomials in $N$ variables and fermionic degree $m$ over the field $\mathbb Q$.

\subsection{Superpartitions}
We  first recall some definitions
related to partitions \cite{Mac}.
A partition $\lambda=(\lambda_1,\lambda_2,\dots)$ of degree $|\lambda|$
is a vector of non-negative integers such that
$\lambda_i \geq \lambda_{i+1}$ for $i=1,2,\dots$ and such that
$\sum_i \lambda_i=|\lambda|$.  
Each partition $\lambda$ has an associated Ferrers diagram
with $\lambda_i$ lattice squares in the $i^{th}$ row,
from the top to bottom. Any lattice square in the Ferrers diagram
is called a cell (or simply a square), where the cell $(i,j)$ is in the $i$th row and $j$th
column of the diagram.  
The conjugate $\lambda'$ of  a partition $\lambda$ is represented  by
the diagram
obtained by reflecting  $\lambda$ about the main diagonal.
We say that the diagram $\mu$ is contained in $\la$, denoted
$\mu\subseteq \la$, if $\mu_i\leq \la_i$ for all $i$.  Finally,
$\la/\mu$ is a horizontal (resp. vertical) $n$-strip if $\mu \subseteq \lambda$, $|\lambda|-|\mu|=n$,
and the skew diagram $\la/\mu$ does not have two cells in the same column
(resp. row). 

Symmetric superpolynomials  are naturally indexed by superpartitions. A superpartition $\Lambda$ of
degree $(n|m)$, or  $\Lambda \vdash (n|m)$ for short, 
 is a pair $(\Lambda^\circledast,\Lambda^*)$ of partitions
$\Lambda^\circledast$ and $\Lambda^*$ such
 that:\\ \\
 1. $\Lambda^* \subseteq \Lambda^\circledast$;\\
 2.  the degree of $\Lambda^*$ is $n$;\\
 3.  the skew diagram $\Lambda^\circledast/\Lambda^*$
is both a horizontal and a vertical $m$-strip\footnote{Such diagrams are sometimes called $m$-rook strips.}\\ \\
We refer to  $m$ and $n$ respectively as the fermionic degree 
and total degree 
of $\La$. 
 Obviously, if
$\Lambda^\circledast= \Lambda^*=\lambda$,
then $\Lambda=(\lambda,\lambda)$ can be interpreted as the partition
$\lambda$.

We will also need another characterization of a superpartition.
 A superpartition $\La$ is 
a pair of partitions $(\La^a; \La^s)=(\La_{1},\ldots,\La_m;\La_{m+1},\ldots,\La_N)$, 
 where $\La^a$ is a partition with $m$ 
distinct parts (one of them possibly  equal to zero),
and $\La^s$ is an ordinary partition (with possibly a string of zeros at the end).   The correspondence 
between $(\Lambda^\circledast,\Lambda^*)$ and 
$(\Lambda^a; \Lambda^s)$ is given explicitly as follows:
given 
$(\Lambda^\circledast,\Lambda^*)$, the parts of $\Lambda^a$ correspond to the
parts of $\Lambda^*$ such that $\Lambda^{\circledast}_i\neq 
\Lambda^*_i$, while the parts of $\Lambda^s$ correspond to the
parts of $\Lambda^*$ such that $\Lambda^{\circledast}_i=\Lambda^*_i$.

The conjugate of a superpartition 
$\Lambda=(\Lambda^\circledast,\Lambda^*)$ is $\Lambda'=((\Lambda^\circledast)',(\Lambda^*)')$.
A diagrammatic representation of $\La$ is given by 
the Ferrers diagram of $\La^*$ with circles added in the cells corresponding
to $\Lambda^{\circledast}/\Lambda^*$.
For instance, if $\La=(\Lambda^a;\Lambda^s)
=(3,1,0;2,1)$,  we have $\Lambda^\circledast=(4,2,2,1,1)$ and $\Lambda^*
=(3,2,1,1)$, so that 
\begin{equation} \label{exdia}
{\scriptsize     \La^\cd:\quad{\tableau[scY]{&&&\\&\\&\\\\ \\ }} \quad
         \La^*:\quad{\tableau[scY]{&&\\&\\ \\ \\ }} \quad
 \Longrightarrow\quad      \La:\quad {\tableau[scY]{&&&\bl\tcercle{}\\&\\&\bl\tcercle{}\\ \\
    \bl\tcercle{}}} \quad    \La':\quad {\tableau[scY]{&&&&\bl\tcercle{}\\&&\bl\tcercle{}\\ \\
    \bl\tcercle{}}},}
\end{equation}
where the last diagram illustrates the conjugation operation that corresponds, as usual, to replacing rows by columns.

We say that $\Omega/\Lambda$ is a horizontal (resp. vertical) $n$-strip if $\Omega^*/\Lambda^*$ and $\Omega^\circledast/\Lambda^\circledast$
are both horizontal (resp. vertical) $n$-strips. Similarly, we say that $\Omega/\Lambda$ is a horizontal
 (resp. vertical) $\tilde n$-strip if $\Omega^*/\Lambda^*$ 
and $\Omega^\circledast/\Lambda^\circledast$ are respectively
a horizontal (resp. vertical) $n$-strip and a horizontal (resp. vertical) $(n+1)$-strip.  For example, using 
\begin{equation}
\Lambda= \, \, {\scriptsize
\tableau[scY]{  &&&&\\ &&& \\  &&& \bl \tcercle{}\\ &&  \\ & \bl \tcercle{} \\ } } \qquad {\rm and} \qquad \Omega= \, \, {\scriptsize \tableau[scY]{  &&&&\\ &&& & \\  &&& & \bl \tcercle{} \\ &&&  \\ & \bl \tcercle{} \\  \bl \tcercle{}}}
\end{equation}
we have that $\Omega/\Lambda$ is a vertical $\tilde 3$-strip since
\begin{equation} \label{exstrip}
  \Omega^*/\Lambda^* = \, \,
  {\scriptsize \tableau[scY]{ \fl &\fl&\fl&\fl&\fl\\ \fl&\fl&\fl&\fl &  \\  \fl&\fl&\fl&  \\ \fl&\fl&\fl &  \\ \fl  \\  } } 
 \qquad {\rm and} \qquad \Omega^\circledast/\Lambda^\circledast = \,  \, {\scriptsize \tableau[scY]{  \fl&\fl&\fl&\fl&\fl\\ \fl&\fl&\fl&\fl & \\  \fl&\fl&\fl&\fl & \\ \fl&\fl&\fl & \\\fl &\fl \\  \\}}
 \end{equation}
form a vertical $3$-strip and a vertical $4$-strip respectively.

From the dominance ordering on partitions
\begin{equation}\label{ordre}
   \mu \leq \la\quad\text{ iff }\quad |\mu|=|\la|\quad\text{
and }\quad \mu_1 + \cdots + \mu_i \leq \lambda_1 + \cdots + \lambda_i\quad \forall i \, . \end{equation}
we define the dominance ordering on superpartitions as
\begin{equation} \label{eqorder1}
 \Omega\leq\Lambda \quad \text{iff}\quad
 \deg(\La)=\deg(\Om) ,
 \quad \Omega^* \leq \Lambda^*\quad \text{and}\quad
\Omega^{\circledast} \leq  \Lambda^{\circledast}
\end{equation}
where we stress that the order on partitions is the dominance ordering.

\subsection{Simple bases}
Four simple bases  of the space of 
symmetric polynomials in superspace will be particularly relevant to our work \cite{DLM1}:
\begin{enumerate}\item 
the generalization of the monomial symmetric 
functions, $m_\La$, defined by 
\begin{equation}
m_\La={\sum_{\sigma \in S_N} }' \theta_{\sigma(1)}\cdots\theta_{\sigma(m)} z_{\sigma(1)}^{\La_1}\cdots z_{\sigma(N)}^{\La_N},
\end{equation}
where the sum is over the  permutations of $\{1,\ldots,N\}$ that produce distinct terms, and where
the entries of $(\Lambda_1,\dots,\Lambda_N)$ are those of $\Lambda=(\La^a; \La^s)=(\La_{1},\ldots,\La_m;\La_{m+1},\ldots,\La_N)$;
\item 
 the generalization of the power-sum 
symmetric
functions $p_\La=\tilde{p}_{\La_1}\cdots\tilde{p}_{\La_m}p_{\La_{m+1}}\cdots p_{\La_{\ell}}\, $
\begin{equation} 
\text{~~~{with}~~~}
\tilde{p}_k=\sum_{i=1}^N\theta_iz_i^k\qquad\text{and}\qquad p_r=
\sum_{i=1}^Nz_i^r \, , \text{~~~~for~~} k\geq 0,~r \geq 1
\end{equation}  
where $\ell$ is the length of the partition $\Lambda^{\circledast}$; 

\item 
 the generalization of the elementary 
symmetric
functions $e_\La=\tilde{e}_{\La_1}\cdots\tilde{e}_{\La_m}e_{\La_{m+1}}\cdots e_{\La_{\ell}}$ ,

\begin{equation} \text{~~~where~~~} \tilde{e}_k=m_{(0;1^k)}
\quad \text{and} \quad e_r=m_{(\emptyset;1^r)}
, \text{~~~~for~~} k\geq 0,~r \geq 1;
\end{equation}  
\item 
 the generalization of the homogeneous
symmetric
functions  $h_\La=\tilde{h}_{\La_1}\cdots\tilde{h}_{\La_m}h_{\La_{m+1}}\cdots h_{\La_{\ell}}\, ,$
\begin{equation} \label{defhomogeneous}
  \text{~~~where~~~}  \tilde{h}_k=\sum_{\Lambda \vdash (n|1)} (\Lambda_1+1)m_{\Lambda}
\qquad \text{and} \qquad h_r=\sum_{\Lambda \vdash (n|0)}m_{\Lambda}, 
 \text{~~~~for~~} k\geq 0,~r \geq 1
\end{equation}  
\end{enumerate}
Observe that when $\Lambda=(\emptyset; \lambda)$, we have that $m_\Lambda=m_\lambda$,  $p_\Lambda=p_\lambda$,  $e_\Lambda=e_\lambda$ and
 $h_\Lambda=h_\lambda$ are respectively the usual monomial, power-sum, elementary and homogeneous symmetric functions.  Also note
that if we define the operator $d=\theta_1 {\partial}/{\partial_{z_1}}+\cdots + \theta_N \partial/\partial_{z_N}$, we have
\begin{equation}
(k+1) \, \tilde p_k = d ( p_{k+1})\, , \qquad  \tilde e_k = d (e_{k+1})  \quad \text{and} \quad \tilde h_k = d(h_{k+1})  
\end{equation}
that is, the new generators in the superspace versions of the bases can be obtained from acting with $d$ on the generators of the usual symmetric function versions.

\subsection{Jack and Macdonald polynomials in superspace}

The Jack polynomials in superspace $\{P_{\Lambda}^{(\alpha)}\}_{\Lambda}$ can be defined as the unique basis of the space
of symmetric functions in superspace such that
\begin{equation} \label{defjack}
  \begin{split}
& P_{\Lambda}^{(\alpha)} = m_\Lambda + \text{smaller terms} \\
& \langle \langle P_{\Lambda}^{(\alpha)}, P_{\Omega}^{(\alpha)} \rangle \rangle_\alpha = 0 \quad \text{ if } \Lambda \neq \Omega
  \end{split}
  \end{equation}
where the scalar product $\langle \langle \cdot, \cdot \rangle \rangle_\alpha$ is defined on power-sum symmetric functions as
\begin{equation}
\langle \langle p_{\Lambda}, p_{\Omega} \rangle \rangle_\alpha = \delta_{\Lambda \Omega}  \, \alpha^{\ell(\Lambda^s)} z_{\Lambda^s}\qquad \qquad 
z_\lambda =\prod_{i\geq 1} i^{n_i(\lambda)} n_i(\lambda)!
\end{equation}  
with $m$ the fermionic degree of $\Lambda$ and $n_i(\lambda)$ the number of parts equal to $i$ in the partition $\lambda$.

The Macdonald polynomials in superspace $\{P_{\Lambda}^{(q,t)}\}_{\Lambda}$ can be defined similarly by replacing  the scalar product 
$\langle \langle \cdot, \cdot \rangle \rangle_\alpha$ in \eqref{defjack}
by the scalar product $\langle \langle \cdot, \cdot \rangle \rangle_{q,t}$ such that
\begin{equation}
\langle \langle p_{\Lambda}, p_{\Omega} \rangle \rangle_{q,t} = \delta_{\Lambda \Omega} \, q^{|\Lambda^a|} \, z_{\Lambda^s} \prod_{i=1}^{\ell(\Lambda^s)} \frac{1-q^{\Lambda^s_i}}{1-t^{\Lambda^s_i}}
\end{equation}

\section{Pieri rules} \label{SecPieri}
Before describing the Pieri coefficients $v_{\Lambda \Omega}(\alpha)$ and $\tilde v_{\Lambda \Omega}(\alpha)$
defined in \eqref{pierifirst}, we first need to establish some notation.  When describing the vertical strips $\Omega/\Lambda$, we will
denote
\begin{itemize}
\item  the squares of $\Lambda$ (the preexisting squares) by ${\footnotesize \tableau[scY]{ \\   }}$ 
\item  the squares of $\Omega/\Lambda$ that 
do not lie in a circle of $\Lambda$ (the new squares) by ${\footnotesize  \tableau[scY]{ \fl }}$ 
\item  the squares of $\Omega/\Lambda$
that lie over a circle of $\Lambda$ (the bumping squares) by  ${\footnotesize \tableau[scY]{\tf \ocircle}}$ 
\item  the circles of 
$\Lambda$ that are still circles in $\Omega$  (the preexisting circles) by 
 ${\footnotesize \tableau[scY]{\bl \tcercle{}} }$
\item  the circles of $\Omega$ that were not circles in
$\Lambda$ (the new circles) by  ${\footnotesize \tableau[scY]{\bl \cerclep} }$ 
\end{itemize}
For instance, the cells of the vertical strip $\Omega/\Lambda$ given in \eqref{exstrip} are described by
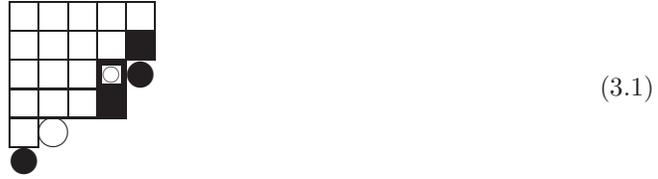
\begin{equation} \label{excells}
{\small \tableau[scY]{  &&&&\\ &&& & \fl \\  &&& \tf \ocircle & \bl \cerclep \\ &&& \fl  \\ & \bl \tcercle{} \\  \bl \cerclep}}
\end{equation}

As is the case for the Pieri rules for the Jack polynomials, the contribution to the coefficients
$v_{\Lambda \Omega}(\alpha)$  and $\tilde v_{\Lambda \Omega}(\alpha)$  will come from the hook-lengths of the cells above 
the non-preexisting cells (the new squares, the bumping squares and the new circles).
Apart from a sign, the contributions will be of two types: a factorized part $\psi_{\Omega/\Lambda}'$
consisting of products of quotients of linear 
factors in $\alpha$ (just as in the usual case\footnote{Note that we use the notation of \cite{Mac} to describe the factorized part $\psi_{\Omega/\Lambda}'$.}) and a determinant ${\rm Det_{\Omega/\Lambda}}$
of a matrix built from the hook-lengths between certain
special cells.

We first describe the factorized part $\psi_{\Omega/\Lambda}' $.  We denote by
${\rm col}_{\Omega/\Lambda}$ the cells of $\Lambda^{\circledast}$ that belong to a column of $\Omega$ containing a non-preexisting cell.   We have that
\begin{equation}
\psi_{\Omega/\Lambda}'  = \prod_{s \in {\rm col}_{\Omega/\Lambda}} c_{\Omega/\Lambda}(s)
\end{equation}
where the contribution $c_{\Omega/\Lambda}(s)$
of a cell $s$ (whose column will contain for illustrative purposes the three possible types of non-preexisting cells) is equal to $A$, $B$, $C$, or 1 according to how its row ends:
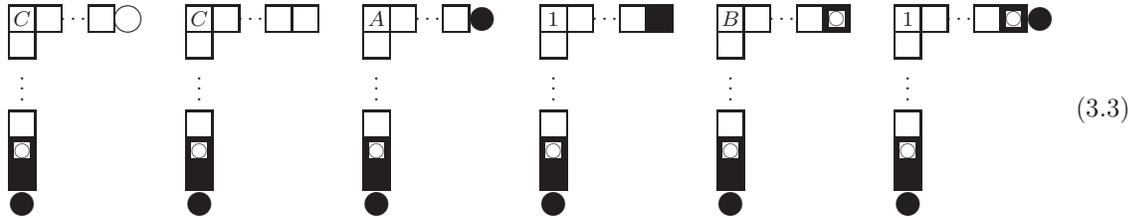
\begin{equation} \label{sixcases}
{\footnotesize\tableau[scY]{ C  &   &\bl $\ldots$  & & \bl \tcercle{} \\  \\ \bl \vspace{-2ex}\vdots \\ \bl& \bl \\  & \bl\\   \tf \ocircle \\  \fl \\  \bl \cerclep     } \qquad  \footnotesize\tableau[scY]{ C  &   &\bl $\ldots$  & &  \\  \\ \bl \vspace{-2ex}\vdots \\ \bl& \bl \\  & \bl\\    \tf \ocircle \\  \fl \\  \bl \cerclep     } \qquad  \footnotesize\tableau[scY]{ A  &   &\bl $\ldots$  & & \bl \cerclep \\  \\ \bl \vspace{-2ex}\vdots \\ \bl& \bl \\  & \bl\\   \tf \ocircle \\  \fl \\  \bl \cerclep     } \qquad \footnotesize\tableau[scY]{ 1  &   &\bl $\ldots$  & & \fl \\  \\ \bl \vspace{-2ex}\vdots \\ \bl& \bl \\  & \bl\\    \tf \ocircle \\  \fl \\  \bl \cerclep    } \qquad  \footnotesize\tableau[scY]{ B  &   &\bl $\ldots$  & & \tf \ocircle \\  \\ \bl \vspace{-2ex}\vdots \\ \bl& \bl \\  & \bl\\    \tf \ocircle \\  \fl \\  \bl \cerclep    } \qquad  \footnotesize\tableau[scY]{ 1  &   &\bl $\ldots$  & & \tf \ocircle & \bl \cerclep  \\  \\ \bl \vspace{-2ex}\vdots \\ \bl& \bl \\  & \bl\\    \tf \ocircle \\  \fl \\  \bl \cerclep     } }
\end{equation}
where
\begin{equation}
A = h_{\Omega/\Lambda}^A(s)= \frac{h_\Lambda^{(\alpha)}(s)}{h_\Omega^{(\alpha)}(s)}, \quad 
B= h_{\Omega/\Lambda}^B(s) =\frac{h^\Omega_{(\alpha)}(s)}{h^\Lambda_{(\alpha)}(s)} \quad {\rm and}  \quad C= h_{\Omega/\Lambda}^C(s)=\frac{h_\Lambda^{(\alpha)}(s)
h^\Omega_{(\alpha)}(s)}{h_\Omega^{(\alpha)}(s)h^\Lambda_{(\alpha)}(s)}
\end{equation}
with
\begin{equation}
h_\Lambda^{(\alpha)}(s)= \ell_{\Lambda^\circledast}(s) + \alpha (a_{\Lambda^*}(s)+1) \qquad {\rm and} \qquad
h^\Lambda_{(\alpha)}(s)= \ell_{\Lambda^*}(s) + 1+ \alpha a_{\Lambda^\circledast}(s) 
\end{equation}
For instance, the contributions of the cells in the following vertical $2$-strip are
\begin{equation}
{\small \tableau[scY]{  1&&&1&1&\fl \\ A &&&A &\bl \cerclep \\  B&&&  \tf \ocircle \\ C&&  \\ C& \bl \tcercle{}  \\ \bl \cerclep}}
\end{equation}
Note that $C=AB$.
We observe that a new circle removes the contribution $B$ in the product $AB$ while a bumping square removes the contribution $A$.  Since a new square can be thought as a combination of a new circle and a bumping square, 
it removes
both contributions $A$ and $B$ from the product $AB$ (leaving 1 as a contribution).  
Note also that, with this rule, a bumping square lying above a new square or a new circle would contribute a factor $B$
to the coefficient  $v_{\Lambda \Omega}$ or $\tilde v_{\Lambda \Omega}$ if it does not have a new circle to its right  (note that $h^\Lambda_{(\alpha)}$ of an empty cell  is defined to be equal to 1).

We now describe the determinant ${\rm Det_{\Omega/\Lambda}}$.  We first do the case
$v_{\Lambda \Omega}$.  Label the bumping squares and new squares from top to bottom by $x_1,\dots,x_n$.  Similarly, 
label the new circles and new squares from top to bottom by $y_1,\dots, y_n$.  Observe that it is natural for a new square,
 being in some sense a  combination of a new circle and a bumping square, to get an
$x$ and a $y$ label. 
If $x_i$  (resp. $y_i$) labels a bumping square or a new square
(resp. a new circle or a new square)  in position  $(a,b)$, we then let $x_i$ (resp. $y_i$) take the value\footnote{Note that for simplicity we will not distinguish between the label $x_i$ and the value $x_i=\alpha b-a$.}
\begin{equation} \label{defab}
x_i= \alpha b-a \qquad ( \text{resp. } y_i= \alpha b-a )
\end{equation}
For example, the labels of the following vertical strip are given by
\begin{equation} \label{eq3.8}
{\small \tableau[scY]{   &&&&&&     \bl \cerclep  &\bl \, y_{1}\\ &&&&&\fl  & \bl \, x_{1} & \bl \, y_{2}  \\ && \tf \ocircle & \bl \cerclep &   \bl \, x_{2} & \bl  \, y_{3} \\ \\ \tf \ocircle  &  \bl \, x_{3} }}
\end{equation}
with for instance $x_1=y_2=6\alpha-2$, $x_2=3\alpha-3$ and $y_3=4\alpha-3$.

Finally, let
\begin{equation} \label{defxy}
[x_i;y_j]_{\alpha} :=
\left \{
\begin{array}{cl}
\dfrac{\alpha}{(x_{i}-y_{j}+\alpha)(x_{i}-y_{j})} & \text{~if there is no~} y  \text{~label in the row of}~x_i \\
1 & \text{~if~} x_i \text{~and~} y_j \text{~are in the same row} \\
0 & \text{~if there is a label~} y_\ell \neq y_j  \text{~in the row of~} x_i 
\end{array} \right .
\end{equation} 
We then have that 
\begin{equation}
{\rm Det_{\Omega/\Lambda}} =
\left|
\begin{array}{cccc}
\displaystyle{[x_1; y_1]_{\alpha}} & \displaystyle{[x_1; y_2]_{\alpha}}  &\cdots & \displaystyle{[x_1; y_n]_{\alpha}}  \\ 
\displaystyle{[x_2; y_1]_{\alpha}}  & \displaystyle{[x_2; y_2]_{\alpha}}  & \cdots & \displaystyle{[x_2; y_n]_{\alpha}}  \\
\vdots& \vdots& \ddots& \vdots \\
\displaystyle{[x_n; y_1]_{\alpha}}  & \displaystyle{[x_n; y_2]_{\alpha}}  & \cdots & \displaystyle{[x_n; y_n]_{\alpha}}  
\end{array}
\right| 
\end{equation}
If we consider the vertical strip in \eqref{eq3.8}, we have for instance
\begin{equation}
{\rm Det_{\Omega/\Lambda}} =
\left|
\begin{array}{cccc}
\displaystyle{[x_1; y_1]_{\alpha}} & \displaystyle{[x_1; y_2]_{\alpha}}  & \displaystyle{[x_1; y_3]_{\alpha}}  \\ 
\displaystyle{[x_2; y_1]_{\alpha}}  & \displaystyle{[x_2; y_2]_{\alpha}}  & \displaystyle{[x_2; y_3]_{\alpha}}  \\
\displaystyle{[x_3; y_1]_{\alpha}}  & \displaystyle{[x_3; y_2]_{\alpha}}  & \displaystyle{[x_3; y_3]_{\alpha}}  
\end{array}
\right| 
= {\left|
\begin{array}{ccc}
0  & 1  & 0 \\ 
0 & 0 & 1  \\
\frac{\alpha}{(5\alpha+4)(6\alpha+4)}  & \frac{\alpha}{(4\alpha+3)(5\alpha+3)}  & \frac{\alpha}{(2\alpha+2)(3\alpha+2)} 
\end{array}
\right| 
}
\end{equation}
We will show in Section~\ref{sec6vertex} that the extension of $
{\rm Det_{\Omega/\Lambda}}$ to the Macdonald case
(which gives back the Jack case in the
$q=t^\alpha$, $t\to 1$ limit)
happens to be a determinant related to the partition function of the 6-vertex model in statistical mechanics.

The case $\tilde v_{\Lambda \Omega}$ is quite similar.
 Label the bumping squares and new squares from top to bottom by $x_1,\dots,x_n$.  Label also the new circles and new squares from top to bottom by $y_1,\dots, y_{n+1}$ (the extra circle comes from
the fact that $\tilde e_n$ increases the fermionic degree).  
Using the notation introduced in \eqref{defab} and \eqref{defxy}, 
we have this time
\begin{equation}
{\rm Det_{\Omega/\Lambda}} =
\left|
\begin{array}{cccc}
1 & 1 & \cdots & 1 \\
\displaystyle{[x_1; y_1]_{\alpha}} & \displaystyle{[x_1; y_2]_{\alpha}}  &\cdots & \displaystyle{[x_1; y_{n+1}]_{\alpha}}  \\ 
\displaystyle{[x_2; y_1]_{\alpha}}  & \displaystyle{[x_2; y_2]_{\alpha}}  & \cdots & \displaystyle{[x_2; y_{n+1}]_{\alpha}}  \\
\vdots& \vdots& \ddots& \vdots \\
\displaystyle{[x_n; y_1]_{\alpha}}  & \displaystyle{[x_n; y_2]_{\alpha}}  & \cdots & \displaystyle{[x_n; y_{n+1}]_{\alpha}}  
\end{array}
\right| 
\end{equation}

We can now state our main theorem.
\begin{theorem} \label{theo} The Pieri rules for the Jack polynomials in superspace are given by
  \begin{equation} \label{eqPieriTheo}
e_n \, P_\Lambda^{(\alpha)} = \sum_\Omega v_{\Lambda \Omega}(\alpha) \, P_\Omega^{(\alpha)} 
\qquad {\rm and} \qquad \tilde e_n \, P_\Lambda^{(\alpha)} = \sum_\Omega \tilde v_{\Lambda \Omega}(\alpha) \, P_\Omega^{(\alpha)} 
  \end{equation}
  where the sum is over all $\Omega$'s such that $\Omega/\Lambda$ is  a vertical $n$-strip and
  a vertical $\tilde n$-strip  respectively.  Moreover,
\begin{equation} \label{eqPieriex}
v_{\Lambda \Omega}(\alpha)= (-1)^{\#(\Omega/\Lambda)}
\psi_{\Omega/\Lambda}' {\rm Det_{\Omega/\Lambda}}  \qquad {\rm and} \qquad
 \tilde v_{\Lambda \Omega}(\alpha)=  (-1)^{\#(\Omega/\Lambda)} \psi_{\Omega/\Lambda}' {\rm Det_{\Omega/\Lambda}} 
\end{equation}
where the quantity ${\#(\Omega/\Lambda)}$ 
stands for the sum 
of the number of preexisting circles and new squares  that lie above each new circle and each bumping square (in other words, it is the length of the shortest permutation necessary to move every new circle and bumping square above the preexisting circles and new squares). 
\end{theorem}
\begin{proof}
  Since the proof is quite long and technical, most of the details will be relegated to the next section (Lemma~\ref{lem0} and Propositions~\ref{mainprop}
  and \ref{mainprop2}).  
  
  The proof proceeds by induction using three quantities belonging to the negative-half of the super Virasoro algebra studied in
\cite{DLM4}.  Consider $Q$ and $q^\perp$ defined as
\begin{equation}
Q = \sum_{i=1}^N \theta_i \left(\frac{N}{\alpha} +z_i \partial_{z_i}\right) \qquad {\rm and} \qquad 
q^\perp = \sum_{i=1}^N z_i \partial_{\theta_i}
\end{equation}  
with $N$ the number of variables (which will soon become irrelevant). The explicit 
action of  $Q$ and $q^\perp$, as well as of $\tilde e_0$, on Jack polynomials in superspace was obtained 
in \cite{DLM4}
\begin{equation}  \label{actione0}
\tilde{e}_{0}P_{\Lambda}^{(\alpha)}=\sum_{\Omega}(-1)^{\#(\Omega/\Lambda)} \left(\prod_{s \in {\rm col}_{\Omega/\Lambda}}\dfrac{h_{\Lambda}^{(\alpha)}(s)}{h_{\Omega}^{(\alpha)}(s)}\right)P_{\Omega}^{(\alpha)} 
\end{equation}
\begin{equation} \label{actionQ}
Q P_{\Lambda}^{(\alpha)}=\sum_{\Omega}(-1)^{\#(\Omega/\Lambda)} \left(\prod_{s \in {\rm col}_{\Omega/\Lambda}}\dfrac{h_{\Lambda}^{(\alpha)}(s)}{h_{\Omega}^{(\alpha)}(s)}\right)\dfrac{(N+1-i+\alpha(j-1))}{\alpha}P_{\Omega}^{(\alpha)} 
\end{equation}
and 
\begin{equation}
q^\perp P_{\Lambda}^{(\alpha)}=\sum_{\Omega}(-1)^{\#(\Omega/\Lambda)} \left(\prod_{s \in {\rm col}_{\Omega/\Lambda}}\dfrac{h^{\Omega}_{(\alpha)}(s)}{h^{\Lambda}_{(\alpha)}(s)}\right)P_{\Omega}^{(\alpha)} 
\end{equation} 
where in $\tilde e_0$ and $Q$ (resp. $q^\perp$) 
the sum is over all $\Omega$'s such that $\Omega/\Lambda$ consists of a new circle (resp. a bumping square).
In \eqref{actionQ}, $(i,j)$ corresponds to the position of the new circle.  In order to get rid of the dependency in $N$,
we use the operator $\tilde Q := Q-N\tilde e_0/\alpha$, whose action is easily seen to be
\begin{equation} \label{actionQt}
\tilde Q P_{\Lambda}^{(\alpha)}=\sum_{\Omega}(-1)^{\#(\Omega/\Lambda)} \left(\prod_{s \in {\rm col}_{\Omega/\Lambda}}\dfrac{h_{\Lambda}^{(\alpha)}(s)}{h_{\Omega}^{(\alpha)}(s)}\right)\dfrac{(1-i+\alpha(j-1))}{\alpha}P_{\Omega}^{(\alpha)} 
\end{equation}
It is straightforward to show that
\begin{equation} \label{commu1}
\{\tilde e_{n-1},q^\perp\} = n \, e_{n} \qquad \forall n \geq 1
\end{equation}
and
\begin{equation} \label{commu2}
\tilde e_0 e_{n} -[\tilde Q,e_{n}]= \tilde e_{n} \qquad \forall n \geq 1
\end{equation}
where $\{a,b\}= ab+ba$ and $[a,b]=ab-ba$.

After showing  that the action of $\tilde e_0$  
is consistent with the statement of the theorem (this is done in Lemma~\ref{lem0}),
we have the base case to start our inductive process.
Observe that knowing the action of $\tilde e_0$
and $q^\perp$ then allows by \eqref{commu1} to get the action of $e_1$. Then knowing the action of 
$e_1$, $\tilde e_0$ and $\tilde Q$ allows by \eqref{commu2} to get the action of $\tilde e_1$.  Repeating these
two steps again and again gives the Pieri rules for $e_n$ and $\tilde e_n$.
Our main (and most difficult) task is thus to show that
\begin{itemize}
\item If the Pieri rules hold for $\tilde e_{n-1}$, then they also hold for
  $e_n = \{\tilde e_{n-1},q^\perp\}/n$.   
\item If the Pieri rules hold for $e_{n}$, then they also hold for
$\tilde e_n=\tilde e_0 e_{n} -[\tilde Q,e_{n}]$.
\end{itemize}
These statements, which will be proven in Propositions~\ref{mainprop} and \ref{mainprop2}, imply by induction that the Pieri rules for the Jack polynomials in superspace hold.
\end{proof}

\begin{remark}  \label{remark} The general case can always be obtained as a limit of the special case where labels $x$ and $y$ are never in the
  same row.  This is easily seen as follows.  Supposing for instance that $\Omega/\Lambda$ is a vertical $\tilde n$-strip,
let
\begin{equation}
 {\rm Det_{\Omega/\Lambda}'}=
\left|
\begin{array}{cccc}
1 & 1 & \cdots & 1 \\
\displaystyle{[x_1; y_1]'_{\alpha}} & \displaystyle{[x_1; y_2]'_{\alpha}}  &\cdots & \displaystyle{[x_1; y_{n+1}]'_{\alpha}}  \\ 
\displaystyle{[x_2; y_1]'_{\alpha}}  & \displaystyle{[x_2; y_2]'_{\alpha}}  & \cdots & \displaystyle{[x_2; y_{n+1}]'_{\alpha}}  \\
\vdots& \vdots& \ddots& \vdots \\
\displaystyle{[x_n; y_1]'_{\alpha}}  & \displaystyle{[x_n; y_2]'_{\alpha}}  & \cdots & \displaystyle{[x_n; y_{n+1}]'_{\alpha}}  
\end{array}
\right|
\end{equation}
where
\begin{equation}
[x_i;y_j]'_\alpha = \frac{\alpha}{(x_i-y_j+\alpha)(x_i-y_j)} 
\end{equation}
We will now see that  
\begin{equation} \label{newBB}
  \lim_{\substack{x_{I_\ell}=y_{J_\ell}\\ x_{i_r}=y_{j_r}-\alpha}} \prod (y_{j_r} -x_{i_r}-\alpha) \prod (x_{I_\ell}-y_{J_\ell}) \,  {\rm Det_{\Omega/\Lambda}'}=
      {\rm Det_{\Omega/\Lambda}}
\end{equation}
where the pairs $(x_{i_r}, y_{j_r})$ are those such that $x_{i_r}= y_{j_r}-\alpha$ (that is, those such that $x_{i_r}$ labels
a cell just to the left of the cell labeled by $y_{j_r}$)
while the pairs $(x_{I_\ell}, y_{J_\ell})$
are those such that  $x_{I_\ell}= y_{J_\ell}$ (that is, those such that
$x_{I_\ell}$ and  $y_{J_\ell}$ label  cells in the same position).
In the first case, we have
$$
\lim_{x_i=y_\ell-\alpha} (y_{\ell}-x_i-\alpha) \, [x_i;y_\ell]'_{\alpha}= 1 \qquad
{\rm and} \qquad
\lim_{x_i=y_\ell-\alpha} (y_{\ell}-x_i-\alpha) \, [x_i;y_j]'_{\alpha}= 0 \quad {\rm for } \quad j\neq \ell
$$
which implies that
$$
\lim_{x_i=y_\ell-\alpha} (y_{\ell}-x_i-\alpha) \, [x_i;y_j]'_{\alpha} =  [x_i;y_j]_{\alpha} \quad {\rm for~all~} j
$$
Similarly, in the other case, we have
$$
\lim_{x_i=y_\ell} (x_i-y_\ell) \, [x_i;y_\ell]'_{\alpha}= 1
\qquad 
{\rm and}
\qquad
\lim_{x_i=y_\ell} (x_i -y_{\ell}) \, [x_i;y_j]'_{\alpha}= 0 \quad {\rm for } \quad j\neq \ell
$$
which implies again that
$$
\lim_{x_i=y_\ell} (x_i-y_{\ell}) \, [x_i;y_j]'_{\alpha} =  [x_i;y_j]_{\alpha} \quad {\rm for~all~} j
$$
We thus have proven that \eqref{newBB} holds since in the rows where $x_i$ is alone (that is, when there is 
no $y_j$ in its row), we have that $ [x_i;y_j]'_{\alpha}= [x_i;y_j]_{\alpha}$ for all $j$.

Using the vertical strip given in \eqref{eq3.8}, we have for instance
\begin{align}
&   \lim_{\substack{x_{1}=y_{2}\\ x_{2}=y_{3}-\alpha}}  (y_{3} -x_{2}-\alpha)  (x_{1}-y_{2}) \,  {\left|
\begin{array}{ccc}
\frac{\alpha}{(x_1-y_1+\alpha)(x_1-y_1)}  & \frac{\alpha}{(x_1-y_2+\alpha)(x_1-y_2)}  & \frac{\alpha}{(x_1-y_3+\alpha)(x_1-y_3)} \\  \\ 
\frac{\alpha}{(x_2-y_1+\alpha)(x_2-y_1)}  & \frac{\alpha}{(x_2-y_2+\alpha)(x_2-y_2)}  & \frac{\alpha}{(x_2-y_3+\alpha)(x_2-y_3)} \\  \\
\frac{\alpha}{(x_3-y_1+\alpha)(x_3-y_1)}  & \frac{\alpha}{(x_3-y_2+\alpha)(x_3-y_2)}  & \frac{\alpha}{(x_3-y_3+\alpha)(x_3-y_3)} 
\end{array}
\right|
  }  \nonumber \\
  & \nonumber \\
& \qquad \qquad \qquad=
      {\left|
\begin{array}{ccc}
0  & 1  & 0 \\ 
0 & 0 & 1  \\
\frac{\alpha}{(x_3-y_1+\alpha)(x_3-y_1)}  & \frac{\alpha}{(x_3-y_2+\alpha)(x_3-y_2)}  & \frac{\alpha}{(x_3-y_3+\alpha)(x_3-y_3)} 
\end{array}
\right| 
} = {\rm Det_{\Omega/\Lambda}}
\end{align}
\end{remark}

\section{Proofs} \label{secproofs}

\subsection{Elementary results}
The following straightforward observations will prove useful.
\begin{lemma} \label{lemma1}
  Suppose that $\Lambda \subseteq \Gamma \subseteq \Omega$.  We have that
\begin{equation}
h_{\Omega/\Lambda}^D(s)=  h_{\Omega/\Gamma}^D (s) \, h_{\Gamma/\Lambda}^D(s) 
\end{equation}
where $D$ stands for $A,B$ or $C$.
\end{lemma}
\begin{lemma} \label{lemmah}
  Let $s$ lie in a column whose only non-preexisting cell is a new circle (resp. a bumping square).  Then
  \begin{equation}
    h_{\Omega/\Lambda}^B(s) =  1 \qquad  \Bigl({\rm resp.~ }  h_{\Omega/\Lambda}^A(s)  \Bigr)
  \end{equation}
and, consequently, 
   \begin{equation}
    h_{\Omega/\Lambda}^C(s) =  h_{\Omega/\Lambda}^A(s)   \qquad  \Bigl({\rm resp.~ }  h_{\Omega/\Lambda}^C(s)= h_{\Omega/\Lambda}^B(s)  \Bigr)
  \end{equation}
  \end{lemma}

\subsection{Preliminary steps.}

In order to start our induction process, we first need to show that \eqref{actione0} coincides with \eqref{eqPieriex} in the case of $\tilde e_0$.
\begin{lemma} \label{lem0}
  The action of $\tilde e_0$ given in \eqref{actione0} coincides with 
the Pieri rule \eqref{eqPieriex} in the case of $\tilde e_0$.
\end{lemma}
\begin{proof}
In the case $\tilde e_0$ of \eqref{eqPieriex},
there is exactly one new circle and no bumping squares or new squares.  The cells $s$ above the new 
circle (in the same column) thus always make a contribution of type $C$, that is,
$$
\psi'_{\Omega/\Lambda} =\prod_{s \in {\rm col}_{\Omega/\Lambda}} h_{\Omega/\Lambda}^C(s)
$$
But $h_{\Omega/\Lambda}^C(s)=  h_{\Omega/\Lambda}^A(s)$ by Lemma~\ref{lemmah}, which gives
$$
\psi'_{\Omega/\Lambda}  =\prod_{s \in {\rm col}_{\Omega/\Lambda}} h_{\Omega/\Lambda}^A(s) = \prod_{s \in {\rm col}_{\Omega/\Lambda}} \frac{h_\Lambda^{(\alpha)}(s)
}{h_\Omega^{(\alpha)}(s)}
$$
as wanted.  Formulas  \eqref{actione0} and \eqref{eqPieriex} then coincide since ${\rm Det_{\Omega/\Lambda}}=1$ in the $\tilde e_0$ case
(without bumping squares or new squares, there are no $x$-labels 
 which leads to the one by one determinant
${\rm Det_{\Omega/\Lambda}}=|1|=1)$.  
\end{proof}

As described in the proof of Theorem~\ref{theo}, we have two results to prove:
(i) supposing that the Pieri rules hold for $\tilde e_{n-1}$ we need to show they also hold
for $e_n$ and (ii)  supposing that the Pieri rules hold for $e_{n}$ we need to show they also hold
for $\tilde e_n$.  We will only do the second case in details since the other case can be shown to
hold in a very similar way.

We first  rewrite \eqref{commu2}
as 
\begin{equation} \label{commu1v2}
\tilde e_{n}= (\tilde e_0-\tilde Q)e_n+e_n \tilde Q
\end{equation}
Using the notation of the Pieri rules \eqref{pierifirst}, 
$\Lambda^+$ will always stand in what follows for $\Lambda$ plus the circle coming from the action of $\tilde Q$ in $e_n\tilde Q$ while $\Omega^-$ will stand for $\Omega$ minus the circle coming from the action 
of $\tilde e_0-\tilde Q$ in $(\tilde e_0-\tilde Q)e_n$. 

We first show that only vertical $\tilde n$-strips can occur.
\begin{lemma} \label{lemmavert}
If the Pieri rules stated in Theorem~\ref{theo} hold for $e_{n}$, then \eqref{commu1v2} implies that in the action of
$\tilde e_n$ on a Jack polynomial in superspace $P^{(\alpha)}_\Lambda$, the only $P_\Omega^{(\alpha)}$'s
that can occur are such that $\Omega/\Lambda$ is a vertical  $\tilde{n}$-strip.
\end{lemma}
\begin{proof}
Since, by hypothesis, $e_n$ only gives rise
to vertical $n$-strips, the only possible non-vertical $\tilde n$-strips are 
such that in exactly one row (dubbed the non-vertical row), the new cells are
of the following two types:
\begin{itemize}
\item[(I)]  $\footnotesize\tableau[scY]{    \fl &  \bl \cerclep     } $ from the action of $(\tilde e_0-\tilde Q)e_n$, where the circle comes from the action of $\tilde e_0-\tilde Q$.
\item[(II)]  $\footnotesize\tableau[scY]{    \tf \bullet &  \bl \cerclep     } $ from the action of $e_n \tilde Q$, where $\tilde Q$ first adds a circle and then 
$e_n$ adds a bumping square and a new circle.
\end{itemize}
 We will show that 
\begin{equation}\label{firstclaim}
\psi'_{\Omega^-/\Lambda} \times \left(\text{ contribution of } \tilde e_0-\tilde Q \right) =  (-1) \times   \left(\text{ contribution of } \tilde Q \right) \times  \psi'_{\Omega/\Lambda^+}
\end{equation}
Then showing that
\begin{equation} \label{secondclaim}
(-1)^{\#(\Omega^-/\Lambda)+\#(\Omega/\Omega^-)} 
{\rm Det}_{\Omega^-/\Lambda} = (-1)^{\#(\Lambda^+/\Lambda)+\#(\Omega/\Lambda^+)} 
{\rm Det}_{\Omega/\Lambda^+} 
\end{equation}
will prove our claim
since the sum of the two contributions will be zero.

We start with \eqref{firstclaim}. We first examine the contribution stemming from the cell
of the new circle in $\tilde Q$ (this contribution is 1 in $\tilde e_0$).  
Supposing that in Type (I) it is position $(i,j)$, we have a contribution from $\tilde e_0-\tilde Q$
of 
$$
1-\frac{(1-i)+\alpha(j-1)}{\alpha}= -\frac{(1-i)+\alpha(j-2)}{\alpha}
$$
in Type (I), while in Type (II) the new circle in $\tilde Q$ is in position $(i,j-1)$ giving a contribution of
$$
\frac{(1-i)+\alpha(j-1-1)}{\alpha}= \frac{(1-i)+\alpha(j-2)}{\alpha}
$$
Comparing the last two equations explains the factor $(-1)$ in the RHS of \eqref{firstclaim}.

 We will now show that for any remaining cell $s$ that contributes, we have
\begin{equation}\label{firstclaim2}
c_{\Omega^-/\Lambda}(s) \times \left(\text{ contribution of } \tilde e_0-\tilde Q \right) =     \left(\text{ contribution of } \tilde Q \right) \times  c_{\Omega/\Lambda^+}(s)
\end{equation}
which will finish the proof of \eqref{firstclaim}.

The possible contributions of the 
cells in the non-vertical row are the same in $\psi'_{\Omega^-/\Lambda}$ and  $\psi'_{\Omega/\Lambda^+}$
since they only contribute a value $1$.  Hence \eqref{firstclaim2} holds for those cells since
$\tilde e_0-\tilde Q$ and $\tilde Q$ do not contribute (which amounts to say that they contribute a factor 1). 

We finally consider the contributions of the
cells in the two columns above $\footnotesize\tableau[scY]{    \fl &  \bl \cerclep     }$ and
 $\footnotesize\tableau[scY]{    \tf \bullet &  \bl \cerclep     } $.  We always consider the contributions 
of the two cells in a given row, which we will call the left and right cells.  There are many cases to consider:
\begin{enumerate}
\item The cells belong to a row with only preexisting cells
\item The cells belong to a row with a new square or with both a bumping square and a new circle
\item The cells  belong to a row with a bumping square
\item The cells  belong to a row with a new circle
\end{enumerate}
We will suppose without loss of generality that there are no cells below the square in  $\footnotesize\tableau[scY]{    \fl &  \bl \cerclep     } $ or $\footnotesize\tableau[scY]{    \tf \bullet &  \bl \cerclep     } $ (using Lemma~\ref{lemma1}, with
$\Gamma$ equal to $\Omega$ without its cells below
the square, ensures that the remaining contribution $h^D_{\Omega/\Gamma}(s)$ is the same in both cases).  Here is the analysis of why \eqref{firstclaim2} holds in each case.  Note that we will use Lemma~\ref{lemmah} again and again without stating it.

Case (1): In Type (I), the left cell contributes $C=AB$ while the right cell contributes $A'$.  In Type (II),
the left cell contributes $A$ from the action of $\tilde Q$ and $B$ from the action of $e_n$ 
while
the right cell contributes $A'B'=A'$.
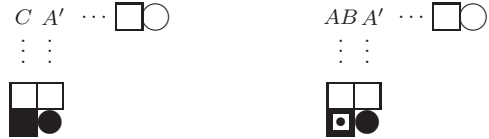
$$
{ \footnotesize\tableau[scY]{ \bl \mbox{$C$}  & \bl \mbox{  $A'$}  &\bl&\bl $\ldots \, $  & & \bl \tcercle{}  \\ \bl \vspace{-0ex}\vdots &\bl \vspace{-0ex}\vdots  \\ \bl& \bl \\  & \\     \fl &  \bl \cerclep   } }\qquad  \qquad \qquad
{ \footnotesize\tableau[scY]{\bl \mbox{$AB$}  &  \bl \mbox{ $\, A'$}  &\bl&\bl $\ldots \, $  & & \bl \tcercle{}  \\ \bl \vspace{-0ex}\vdots &\bl \vspace{-0ex}\vdots  \\ \bl& \bl \\  & \\    \tf \bullet &  \bl \cerclep   }
}
$$

Case (2): In Type (I), the left cell contributes 1 while the right cell contributes $A'$.  In Type (II),
the left cell contributes $A$ from the action of $\tilde Q$ and 1 from the action of $e_n$ while
the right cell contributes 1.  But $A=A'$ since the new square counts in $A'$ but not in $A$. 
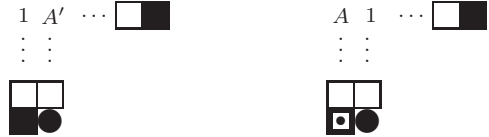
$$
{ \footnotesize\tableau[scY]{ \bl \mbox{$1$}  & \bl \mbox{  $A'$}  &\bl&\bl $\ldots \, $  & & \fl  \\ \bl \vspace{-0ex}\vdots &\bl \vspace{-0ex}\vdots  \\ \bl& \bl \\  & \\     \fl &  \bl \cerclep   } }\qquad  \qquad \qquad
{ \footnotesize\tableau[scY]{\bl \mbox{$A$}  &  \bl \mbox{ $1$}  &\bl&\bl $\ldots \, $  & & \fl  \\ \bl \vspace{-0ex}\vdots &\bl \vspace{-0ex}\vdots  \\ \bl& \bl \\  & \\    \tf \bullet &  \bl \cerclep   }
}
$$

Case (3): In Type (I), the left cell contributes $B$ while the right cell contributes $A'$.  
 In Type (II),
the left cell contributes $A$ from the action of $\tilde Q$ and $B$ from the action of $e_n$ while
the right cell contributes 1.  Again $A=A'$ since the bumping square counts in $A'$ but not in $A$ (circles do not contribute to the arm in $A$). 
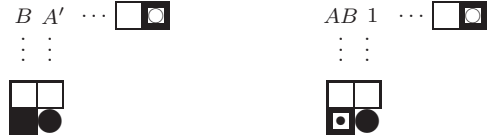
$$
{ \footnotesize\tableau[scY]{ \bl \mbox{$B$}  & \bl \mbox{  $A'$}  &\bl&\bl $\ldots \, $  & & \tf \ocircle  \\ \bl \vspace{-0ex}\vdots &\bl \vspace{-0ex}\vdots  \\ \bl& \bl \\  & \\     \fl &  \bl \cerclep   } }\qquad  \qquad \qquad
{ \footnotesize\tableau[scY]{\bl \mbox{$AB$}  &  \bl \mbox{ $\, 1$}  &\bl&\bl $\ldots \, $  & & \tf \ocircle  \\ \bl \vspace{-0ex}\vdots &\bl \vspace{-0ex}\vdots  \\ \bl& \bl \\  & \\    \tf \bullet &  \bl \cerclep   }
}
$$

Case (4):   In Type (I), the left cell contributes $A$ while the right cell contributes $A'$.  
 In Type (II),
the left cell contributes $A$ from the action of $\tilde Q$ and $1$ from the action of $e_n$ 
 while
the right cell contributes $A'$ (the new circle does not affect the arm in $A$ and $A'$).
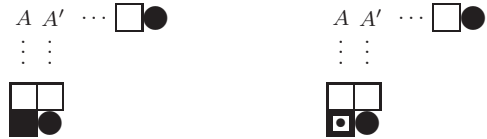
$$
{ \footnotesize\tableau[scY]{ \bl \mbox{$A$}  & \bl \mbox{  $A'$}  &\bl&\bl $\ldots \, $  & & \bl \cerclep  \\ \bl \vspace{-0ex}\vdots &\bl \vspace{-0ex}\vdots  \\ \bl& \bl \\  & \\     \fl &  \bl \cerclep   } }\qquad  \qquad \qquad
{ \footnotesize\tableau[scY]{\bl \mbox{$A$}  &  \bl \mbox{ $A'$}  &\bl&\bl $\ldots \, $  & & \bl \cerclep  \\ \bl \vspace{-0ex}\vdots &\bl \vspace{-0ex}\vdots  \\ \bl& \bl \\  & \\    \tf \bullet &  \bl \cerclep   }
}
$$

We now prove \eqref{secondclaim}.  Suppose that the label of the $x$ and $y$ variables in 
the non-vertical row of 
${\rm Det}_{\Omega^-/\Lambda}$ and ${\rm Det}_{\Omega/\Lambda^+}$ is $(x_i,y_j)$ (the labels are the same in both determinant
since the non-vertical row contains an $x$ and a $y$ label).  The only difference between
${\rm Det}_{\Omega^-/\Lambda}$ and ${\rm Det}_{\Omega/\Lambda^+}$ is that $y_j$ differs in the two determinants.
But by construction, ${\rm Det}_{\Omega^-/\Lambda}$ and ${\rm Det}_{\Omega/\Lambda^+}$ do not depend on $y_j$.
This is seen as follows (only doing the case ${\rm Det}_{\Omega^-/\Lambda}$, the other case being similar): the quantity ${\rm Det}_{\Omega^-/\Lambda}$ is the determinant of a matrix
 whose $i$-th row is made of zeros except in column $j$ where the entry is 1. 
 Hence ${\rm Det}_{\Omega^-/\Lambda}$ is equal, up to a sign, to its $(i,j)$-th minor which does not depend on 
$y_j$ (the only dependency in $y_j$ in the matrix is in column $j$).  Doing the same analysis for
${\rm Det}_{\Omega/\Lambda^+}$ implies that
${\rm Det}_{\Omega^-/\Lambda}$ and ${\rm Det}_{\Omega/\Lambda^+}$ are equal.

We finally have to show that the signs 
coincide.  For $(-1)^{\#(\Omega^-/\Lambda)+\#(\Omega/\Omega^-)}$ and $(-1)^{\#(\Lambda^+/\Lambda)+\#(\Omega/\Lambda^+)}$ we
only consider the contribution of (or what is affected by) the cells in the non-vertical row since the rest of the contributions
are equal on both sides.  The contribution of the non-vertical row to $(-1)^{\#(\Omega^-/\Lambda)+\#(\Omega/\Omega^-)}$
is
$$
(-1)^{\# (\text{preexisting circles above)}+\#(\text{new circles above})+\# (\text{bumping squares below})+\# (\text{new circles below})}
$$
given that the new square 
in that row contributes
$$
(-1)^{\# (\text{bumping squares below})+\# (\text{new circles below})}
$$
while the new circle contributes
$$
(-1)^{\# (\text{preexisting circles above)}+\#(\text{new circles above})}
$$

In the other case, the contribution of the non-vertical row to $(-1)^{\#(\Lambda^+/\Lambda)+\#(\Omega/\Lambda^+)}$ is
$$
(-1)^{\# (\text{preexisting circles above)}+\# (\text{bumping squares above}) }
$$
since the bumping squares above the new circle stemming from the action of $\tilde Q$ were preexisting circles before $e_n$ acted.

If we multiply
the two contributions, we get
$$
(-1)^{\# (\text{new circles below}) +\# (\text{new circles above}) + \# (\text{bumping squares below})+\# (\text{bumping squares above})}=1
$$
because the total number of new circles is equal to the total number of bumping squares in 
$e_n$.
\end{proof}

\subsection{An identity}  Before embarking on the main proof that knowing that $e_n$ obeys the Pieri rules
implies that $\tilde e_n$ also obeys the Pieri rules, we prove an identity that will prove essential.
\begin{lemma} \label{lemiden} We have
  \begin{equation} \label{eqiden}
    \begin{split}
    & \left(1-  \frac{y+1-\alpha}{\alpha}\right) \left( \prod_{i=1}^n \frac{x_i-y+\alpha-1}{x_i-y+\alpha} \right)  
  + 
    \left(\frac{y+1-\alpha}{\alpha}\right) \left( \prod_{i=1}^n \frac{x_i-y-1}{x_i-y} \right)\\
  &  \qquad \qquad \qquad = 1 - \sum_{i=1}^n \frac{\alpha}{(x_i-y+\alpha)(x_i-y)}\left( \frac{x_i+1-\alpha}{\alpha} \right) \left( \prod_{k \neq i} \frac{x_k-x_i-1}{x_k-x_i}  \right)
 \end{split}
  \end{equation}
\end{lemma}
\begin{proof}
  Let the l.h.s. and the r.h.s. of \eqref{eqiden} be respectively $A$ and $B$.
  Now consider
  \begin{equation}
    F(y) = A \times \prod_{k=1}^n (x_k-y+\alpha)(x_k-y) \quad \text{and} \quad  G(y) = B  \times \prod_{k=1}^n (x_k-y+\alpha)(x_k-y)
  \end{equation}
  Showing that $F(y)=G(y)$ thus amounts to proving \eqref{eqiden}.
Observe that $\prod_{k=1}^n (x_k-y+\alpha)(x_k-y)$ is a monic polynomial in $y$ of degree $2n$ over the field $\mathbb Q(x_1,\dots,x_n,\alpha)$.
It is therefore immediate that $G(y)$ is a monic polynomial in $y$ of degree $2n$.

We will now show that
$F(y)$ is also a monic polynomial in $y$ of degree $2n$.  Since
\begin{equation}
\left( \prod_{k=1}^n (x_k-y+\alpha)(x_k-y) \right) \left(  \prod_{i=1}^n \frac{x_i-y+\alpha-1}{x_i-y+\alpha}  \right) =  \prod_{k=1}^n (x_k-y+\alpha-1)(x_k-y)
\end{equation}  
is a monic polynomial in $y$ of degree $2n$, it suffices to show that
  \begin{equation} \label{eqbrack}
 \frac{y+1-\alpha}{\alpha} \left[  \prod_{i=1}^n (x_i-y+\alpha-1)(x_i-y) -   \prod_{i=1}^n (x_i-y-1)(x_i-y+\alpha)  \right]
  \end{equation}  
  is a polynomial in $y$ of degree at most $2n-1$.  But this is immediate given that
  \begin{equation}
(x-y+\alpha-1)(x-y)  =(x-y-1)(x-y+\alpha) +\alpha 
 \end{equation}   
  implies that the term between brackets in \eqref{eqbrack} is of degree at most $2n-2$ in $y$.

  Given that $F(y)$ and $G(y)$ are both monic of degree $2n$, in order to conclude that they are equal
  we only need to verify that they
  coincide
  at the $2n$ points $y=x_\ell+\alpha$ and $y=x_\ell$ for $\ell=1,\dots,n$.  It is straightforward to see that
  \begin{equation}
    \begin{split}
         F(x_\ell+\alpha)= &\left(1- \frac{x_\ell+1}{\alpha} \right) \Biggl[\prod_{i=1}^n (x_i -x_\ell-1)(x_i-x_\ell-\alpha)  \Biggr] \\
       & \qquad \qquad   =  \alpha \left(\frac{\alpha-x_\ell-1}{\alpha} \right) \Biggl[\prod_{k\neq \ell} (x_k -x_\ell-1)(x_k-x_\ell-\alpha)  \Biggr]= G(x_\ell+\alpha)
    \end{split}
  \end{equation}  
  Similarly, it can be checked that
  \begin{equation}
    \begin{split}
         F(x_\ell)= &\frac{x_\ell+1-\alpha}{\alpha}  \Biggl[\prod_{i=1}^n (x_i -x_\ell+\alpha)(x_i-x_\ell-1)  \Biggr] \\
       & \qquad \qquad   =  -\alpha \left(\frac{x_\ell+1-\alpha}{\alpha} \right) \Biggl[\prod_{k\neq \ell} (x_k -x_\ell+\alpha)(x_k-x_\ell-1)  \Biggr]= G(x_\ell)
    \end{split}
  \end{equation}   
  which completes the proof.
  \end{proof}

\subsection{First proposition}
We now prove the induction step for the Pieri rule in the
case of $\tilde e_n$. 
\begin{proposition} \label{mainprop}
If the Pieri rules stated in Theorem~\ref{theo} hold for $e_{n}$, then they also hold for
$\tilde e_n$.
\end{proposition}
\begin{proof}
We have shown in Lemma~\ref{lemmavert} that only vertical  $\tilde{n}$-strips can occur.
We thus have left to show that
if $\Omega/\Lambda$ is a vertical  $\tilde{n}$-strip, then \eqref{commu1v2}
implies by induction that
\begin{equation} \label{toprove}
 \tilde v_{\Lambda \Omega}=  (-1)^{\#(\Omega/\Lambda)} \psi_{\Omega/\Lambda}' {\rm Det_{\Omega/\Lambda}} 
\end{equation}

\noindent{\it Factorized term.} 
We will first show that the factorized term $\psi_{\Omega/\Lambda}'$ is as desired. In order to do so,
we will show that the contribution of a cell 
$s$ coming
from $(\tilde e_0 - \tilde Q)e_n$ and  $e_n \tilde Q$  contains 
$c_{\Omega/\Lambda}(s)$ (as wanted) plus possibly an extra contribution that we will keep track of.  
They will then be used to construct $(-1)^{\#(\Omega/\Lambda)}{\rm Det_{\Omega/\Lambda}}$.  The extra contributions are as follows:
\begin{enumerate}
\item[({\bf C1})] A cell $s$  above the new circle in $\tilde e_0 - \tilde Q$ has an extra contribution
of $h^A_{\Omega/\Omega^-}(s)$ if there is a bumping square or a new square in its row
\item[({\bf C2})] A cell $s$ above the new circle in $\tilde Q$  has an extra contribution
of $h^A_{\Lambda^+/\Lambda}(s)$ if there is a bumping square or a new square in its row
\item[({\bf C3})] A cell $s$ in the row of the new circle in  $\tilde e_0 - \tilde Q$ has an extra contribution
of $h^B_{\Omega^-/\Lambda}(s)$ if there is a  bumping square or a new square in its column
\item[({\bf C4})] A cell $s$ in the row of the new circle in $\tilde Q$ has an extra contribution
of $h^B_{\Omega/\Lambda^+}(s)$ if there is a  bumping square or a new square in its column
\end{enumerate}
We now prove that once the factorized term $c_{\Omega/\Lambda}(s)$ has been extracted, we are left with
the four aforementioned extra contributions.

{\bf (C1)}.  Consider a cell $s$ in the column above the new circle in 
$\tilde e_0 - \tilde Q$ (which acts after $e_n$).  Recall that the contribution of  $\tilde Q$ is that
of $\tilde e_0$, which we have seen corresponds to the six possible hooks in \eqref{sixcases}.  
The contribution of $\tilde e_0-\tilde Q$ is of the form $A=C$ and corresponds to $h_{\Omega/\Omega^-}^A(s)$, 
while the contribution of
$e_n$ is $c_{\Omega^-/\Lambda}(s)$.

Using Lemma~\ref{lemma1} and Lemma~\ref{lemmah}, 
the first three possible hooks 
(those that have at most a new circle in their row) thus give
$$c_{\Omega/\Lambda}(s) = h_{\Omega/\Lambda}^D(s) = h_{\Omega/\Omega^-}^D(s) \, 
h_{\Omega^-/\Lambda}^D(s) = h_{\Omega/\Omega^-}^A(s) \, h_{\Omega^-/\Lambda}^D(s)=
 h_{\Omega/\Omega^-}^A(s) \,   
 c_{\Omega^-/\Lambda}(s) $$ where $D$ stands for $A$ or $C$.  Those hooks do not give extra contributions since the contribution 
 of $\tilde e_0-\tilde Q$ was used to obtain $c_{\Omega/\Lambda}(s)$.

 The two hooks where the contribution is 1 work trivially and give an extra factor 
 $h_{\Omega/\Omega^-}^A(s)$ coming from  the action of $\tilde e_0-\tilde Q$
 (those have a new square or both a bumping square and a new circle in their row).

In the hook whose contribution is $B$ (with a bumping square in its row), 
the circle below $s$ is not seen by the leg and 
thus $c_{\Omega/\Lambda}(s)=c_{\Omega^-/\Lambda}(s)$ as wanted.  We are thus also left with an extra factor $h_{\Omega/\Omega^-}^A(s)$ in that case.

{\bf (C2)}. Consider now a cell $s$ in the column above the new circle in 
$\tilde Q$ (which acts this time before $e_n$).  The contribution of $\tilde Q$ is of the form $A=C$ and corresponds to $h_{\Lambda^+/\Lambda}^A(s)$, 
while the contribution of
$e_n$ is $c_{\Omega/\Lambda^+}(s)$.

Using Lemma~\ref{lemma1} and Lemma~\ref{lemmah},  the first three possible hooks (those that have at most a new circle in their row) are
give
$$c_{\Omega/\Lambda}(s)=  h_{\Omega/\Lambda^+}^D(s) \,  h_{\Lambda^+/\Lambda}^D(s) =  h_{\Omega/\Lambda^+}^D(s) \,  h_{\Lambda^+/\Lambda}^A(s) =  c_{\Omega/\Lambda^+}(s) \,  h_{\Lambda^+/\Lambda}^A(s) $$
and have thus no extra contributions. 

 The two hooks where the contribution is 1 again work trivially and give an extra factor $h_{\Lambda^+/\Lambda}^A(s)$ (those have a new square or a bumping square and a new circle in their row).
 
In the hook whose contribution is $B$ (with a bumping square in its row), 
the circle below $s$ is not seen by the leg and 
thus $c_{\Omega/\Lambda}(s)=c_{\Omega/\Lambda^+}(s)$ as wanted.  We are thus also left with an extra contribution $h_{\Lambda^+/\Lambda}^A(s)$  in that case.

{\bf (C3)}. Consider this time a cell $s$ in the row to the left of the new circle in 
$(\tilde e_0 - \tilde Q)$ (which acts after $e_n$). As in Case (1), the contribution of $\tilde e_0-\tilde Q$ is of the form $A=C$ and corresponds to $h_{\Omega/\Omega^-}^A(s)$, 
while the contribution of
$e_n$ is $c_{\Omega^-/\Lambda}(s)$.   
If the row of $s$ does not have a new cell in $\Omega^-/\Lambda$, then the contribution 
from the action of $e_n$ (supposing that there are new cells of $\Omega^-/\Lambda$ in the
column of $s$) is of type $AB$.  We have that $h^A_{\Omega^-/\Lambda}(s)=c_{\Omega/\Lambda}(s)$ since
the new circle in the arm of $s$ does not affect $h^A_{\Omega^-/\Lambda}(s)$.  The extra contribution is
thus $h^B_{\Omega^-/\Lambda}(s)$.  Note that  $h^B_{\Omega^-/\Lambda}(s)=1$ if there is only a new circle in the column of $s$, which means that there needs to be a bumping square or a new square in the column
of $s$ to have an extra contribution.

The only other option is for the row of $s$ to contain a bumping square in $\Omega^-/\Lambda$.
This time, the contribution from the  the action of $e_n$ is of type $B$.   But 
$c_{\Omega/\Lambda}(s)=1$ and thus the  extra factor is again 
$c_{\Omega^-/\Lambda}(s)=h^B_{\Omega^-/\Lambda}(s)$. Note again that  $h^B_{\Omega^-/\Lambda}(s)=1$ if there is only a new circle in the column of $s$, which means that there needs to be a bumping square or a new square in the column
of $s$ to have an extra contribution.

{\bf (C4)}.
Consider finally a cell $s$ in the row to the left of the new circle in 
$\tilde Q$ (which acts before $e_n$). As in Case (2), the contribution of $\tilde Q$ is of the form $A=C$ and corresponds to $h_{\Lambda^+/\Lambda}^A(s)$, 
while the contribution of
$e_n$ is $c_{\Omega/\Lambda^+}(s)$. 
If the row of $s$ does not have a new cell in $\Omega/\Lambda^+$ (it would necessarily 
be a bumping square), then the contribution of cell $s$ from the action of $e_n$ is of type
$AB$.  Again $h^A_{\Omega/\Lambda^+}(s)=c_{\Omega/\Lambda}(s)$ since
the new circle in the arm of $s$ does not affect $h^A_{\Omega/\Lambda^+}(s)$.  The extra factor is
thus $h^B_{\Omega/\Lambda^+}(s)$.   We have that  $h^B_{\Omega/\Lambda^+}(s)=1$ if there is only a new circle in the column of $s$, which means that there needs to be a bumping square or a new square in the column
of $s$ to have an extra contribution.

The only other option is for the row of $s$ to contain a bumping square in $\Omega/\Lambda^+$.
This time, the contribution from the  the action of $e_n$ is of type $B$.   But 
$c_{\Omega/\Lambda}(s)=1$ (there is a new square in the row of $s$ in $\Omega/\Lambda$) 
and thus the  extra factor is again 
$c_{\Omega/\Lambda^+}(s)=h^B_{\Omega/\Lambda^+}(s)$.  We have again that  $h^B_{\Omega/\Lambda^+}(s)=1$ if there is only a new circle in the column of $s$, which means that there needs to be a bumping square or a new square in the column
of $s$ to have an extra contribution.

\noindent{\it Remaining terms (special case).}
Having now extracted $\psi'_{\Omega/\Lambda}$, to prove \eqref{toprove} 
we have left to show that the remaining terms give
$(-1)^{\#(\Omega/\Lambda)}{\rm Det_{\Omega/\Lambda}}$.  We will first do the special 
case where there  are no new squares and no row with a bumping square and a new circle.  
We will see later that the general case can be deduced from that special one.

We first compute the remaining terms in 
 $(\tilde{e}_{0}-\tilde Q)e_{n}$.  Since $(\tilde{e}_{0}-\tilde Q)$ acts after $e_n$ and since it
adds a circle, we need to consider all possible $\Omega^{-}$.  By hypothesis, there are $n+1$ 
new circles whose positions are $y_1,\dots,y_{n+1}$.  Let  $\Omega^{(j)}$ be $\Omega$ without its
circle in position $y_j$.  By {\bf (C1)}, for all $s$ above $y_j$ and in the row of some
$x_i$, we have an extra contribution 
$$
h_{\Omega/\Omega^{(j)}}^A(s)= \frac{x_i-y_j+\alpha-1}{x_i-y_j+\alpha}
$$
By {\bf (C3)} this time, for all  $s$ in the row of  $y_j$ and in the column of some
$x_i$, we have an extra contribution
$$
h^B_{\Omega^{(j)}/\Lambda}(s)= \frac{y_j-x_i-\alpha+1}{y_j-x_i-\alpha}=\frac{x_i-y_j+\alpha-1}{x_i-y_j+\alpha}
$$
We also need to consider the extra factor coming from the new circle in position $y_j$ in the action of
$(\tilde{e}_{0}-\tilde Q)$, which is given by
$$
1-\frac{y_j+1-\alpha}{\alpha}
$$
The remaining contribution from $\Omega^{(j)}$ in $(\tilde{e}_{0}-\tilde Q)e_{n}$ is thus
\begin{equation} \label{termgen1}
(-1)^{\#(\Omega^{(j)}/\Lambda)+\#(\Omega/\Omega^{(j)})}\left(1-\frac{y_j+1-\alpha}{\alpha}\right) \left( \prod_{i=1}^n \frac{x_i-y_j+\alpha-1}{x_i-y_j+\alpha} \right) {\rm Det}_{\Omega^{(j)}/\Lambda} 
\end{equation}
where 
$$
{\rm Det}_{\Omega^{(j)}/\Lambda} =\left|
\begin{array}{ccccc}
\displaystyle{[x_1; y_1]_{\alpha}}& \cdots & \displaystyle{\widehat{[x_1; y_j]_{\alpha}}}  &\cdots & \displaystyle{[x_1; y_n]_{\alpha}}  \\ 
\displaystyle{[x_2; y_1]_{\alpha}}& \cdots   & \displaystyle{\widehat{[x_2; y_j]_{\alpha}}}  & \cdots & \displaystyle{[x_2; y_n]_{\alpha}}  \\
\vdots& & \vdots& \ddots& \vdots \\
\displaystyle{[x_n; y_1]_{\alpha}}& \cdots  & \displaystyle{\widehat{[x_n; y_j]_{\alpha}}}  & \cdots & \displaystyle{[x_n; y_n]_{\alpha}}  
\end{array}
\right| 
$$
with $\displaystyle{\widehat{[x_\ell; y_j]_{\alpha}}}$ meaning that the term does not exist.  Finally,
$$
(-1)^{\#(\Omega/\Lambda)} = (-1)^{\#(\Omega^{(j)}/\Lambda)+\# (\text{preexisting circles above } y_j)} 
$$
and 
$$
(-1)^{\#(\Omega/\Omega^{(j)})} = (-1)^{\# (\text{preexisting circles above } y_j) + \# (\text{new circles above } y_j )} 
$$
Using the fact that the number of new circles above $y_j$ is $j-1$, we obtain that
\begin{equation} \label{eqsign1}
(-1)^{\#(\Omega^{(j)}/\Lambda)+\#(\Omega/\Omega^{(j)})}= (-1)^{\#(\Omega/\Lambda)+j-1}
\end{equation}
The remaining contribution from $\Omega^{(j)}$ in $(\tilde{e}_{0}-\tilde Q)e_{n}$ is thus
\begin{equation}
(-1)^{\#(\Omega/\Lambda)+j-1}\left(1-\frac{y_j+1-\alpha}{\alpha}\right) \left( \prod_{i=1}^n \frac{x_i-y_j+\alpha-1}{x_i-y_j+\alpha} \right) {\rm Det}_{\Omega^{(j)}/\Lambda} 
\end{equation}
for a total remaining contribution of
\begin{equation} \label{contrib1}
(-1)^{\#(\Omega/\Lambda)} \sum_{j=1}^{n+1} (-1)^{j-1}
\left(1-\frac{y_j+1-\alpha}{\alpha}\right) \left( \prod_{i=1}^n \frac{x_i-y_j+\alpha-1}{x_i-y_j+\alpha} \right) {\rm Det}_{\Omega^{(j)}/\Lambda} 
\end{equation}

We now compute the remaining contribution of $e_n \tilde Q$. Since we are supposing that there are no new squares, the new circle stemming from the action of  $\tilde Q$ will also be a new circle in $\Omega/\Lambda$.  The $n+1$ new circles
in $\Omega/\Lambda$ are labeled again $y_1,\dots,y_{n+1}$.  We let $\Lambda^{(j)}$ be equal to $\Lambda$
with an extra circle in position $y_j$.

By {\bf (C2)}, for all $s$ above $y_j$ and in the row of some
$x_i$, we have an extra contribution
$$
h_{\Lambda^{(j)}/\Lambda}^A(s)= \frac{x_i-y_j-1}{x_i-y_j}
$$
Similarly, by {\bf (C4)},
for all  $s$ in the row of  $y_j$ and in the column of some
$x_i$,
we have an extra contribution
$$
h^B_{\Omega/\Lambda^{(j)}}(s)= \frac{y_j-x_i+1}{y_j-x_i}=\frac{x_i-y_j-1}{x_i-y_j}
$$
We also need to consider the extra factor coming from the new circle in position $y_j$ in the action of
$\tilde Q$, which is given by 
$$
\frac{y_j+1-\alpha}{\alpha}
$$
The remaining contribution from $\Omega^{(j)}$ in $e_{n} \tilde Q$ is thus
\begin{equation}
  \label{termgen2}
(-1)^{\#(\Lambda^{(j)}/\Lambda)+\#(\Omega/\Lambda^{(j)})}\left(\frac{y_j+1-\alpha}{\alpha}\right) \left( \prod_{i=1}^n \frac{x_i-y_j-1}{x_i-y_j} \right) {\rm Det}_{\Omega/\Lambda^{(j)}} 
\end{equation}
with
$$
{\rm Det}_{\Omega/\Lambda^{(j)}} =\left|
\begin{array}{ccccc}
\displaystyle{[x_1; y_1]_{\alpha}}& \cdots & \displaystyle{\widehat{[x_1; y_j]_{\alpha}}}  &\cdots & \displaystyle{[x_1; y_n]_{\alpha}}  \\ 
\displaystyle{[x_2; y_1]_{\alpha}}& \cdots   & \displaystyle{\widehat{[x_2; y_j]_{\alpha}}}  & \cdots & \displaystyle{[x_2; y_n]_{\alpha}}  \\
\vdots& & \vdots& \ddots& \vdots \\
\displaystyle{[x_n; y_1]_{\alpha}}& \cdots  & \displaystyle{\widehat{[x_n; y_j]_{\alpha}}}  & \cdots & \displaystyle{[x_n; y_n]_{\alpha}}  
\end{array}
\right| 
$$
where we stress that ${\rm Det}_{\Omega/\Lambda^{(j)}}$ is equal to  $ {\rm Det}_{\Omega^{(j)}/\Lambda}$. 
 Finally,
$$
(-1)^{\#(\Omega/\Lambda)} = (-1)^{\#(\Omega/\Lambda^{(j)})-\# (\text{new circles below } y_j)-\# (\text{bumping squares below } y_j)+\# (\text{preexisting circles above } y_j)} 
$$
and 
$$
(-1)^{\#(\Lambda^{(j)}/\Lambda)} = (-1)^{\# (\text{preexisting circles above } y_j) +\# (\text{bumping squares above } y_j) } 
$$
Using the fact that the total number of new circles and bumping squares is $2n+1$, we have that
\begin{equation}
\label{eqsign1p}
(-1)^{\#(\Omega/\Lambda^{(j)})+\#(\Lambda^{(j)}/\Lambda)}= (-1)^{\#(\Omega/\Lambda)+\# (\text{new circles above } y_j)}
= (-1)^{\#(\Omega/\Lambda)+j-1}
\end{equation}
Hence, the remaining contribution from $\Omega^{(j)}$ in $e_{n} \tilde Q$ is 
$$
(-1)^{\#(\Omega/\Lambda)+j-1}\left(\frac{y_j+1-\alpha}{\alpha}\right) \left( \prod_{i=1}^n \frac{x_i-y_j-1}{x_i-y_j} \right) {\rm Det}_{\Omega/\Lambda^{(j)}} 
$$
for a total remaining contribution of
\begin{equation} \label{contrib2}
(-1)^{\#(\Omega/\Lambda)} \sum_{j=1}^{n+1} (-1)^{j-1}
\left(\frac{y_j+1-\alpha}{\alpha}\right) \left( \prod_{i=1}^n \frac{x_i-y_j-1}{x_i-y_j} \right) {\rm Det}_{\Omega/\Lambda^{(j)}} 
\end{equation}
Adding \eqref{contrib1} and \eqref{contrib2}, using the fact that ${\rm Det}_{\Omega/\Lambda^{(j)}}= {\rm Det}_{\Omega^{(j)}/\Lambda}$, and comparing with \eqref{toprove}, we see that we have left to show that
\begin{equation}
\begin{split}
 \sum_{j=1}^{n+1} (-1)^{j-1}
 \Biggl[& \left(1-  \frac{y_j+1-\alpha}{\alpha}\right) \left( \prod_{i=1}^n \frac{x_i-y_j+\alpha-1}{x_i-y_j+\alpha} \right)  
 \\
& \qquad+ 
\left(\frac{y_j+1-\alpha}{\alpha}\right) \left( \prod_{i=1}^n \frac{x_i-y_j-1}{x_i-y_j} \right) \Biggr] {\rm Det}_{\Omega/\Lambda^{(j)}}  = {\rm Det}_{\Omega/\Lambda}
\end{split}
\end{equation}
Or more explicitly, that
\begin{equation} \label{bigidentity}
\begin{split}
 \sum_{j=1}^{n+1} (-1)^{j-1}
 \Biggl[& \left(1-  \frac{y_j+1-\alpha}{\alpha}\right) \left( \prod_{i=1}^n \frac{x_i-y_j+\alpha-1}{x_i-y_j+\alpha} \right)  
 \\
& + 
\left(\frac{y_j+1-\alpha}{\alpha}\right) \left( \prod_{i=1}^n \frac{x_i-y_j-1}{x_i-y_j} \right) \Biggr] \left|
\begin{array}{ccccc}
\displaystyle{[x_1; y_1]_{\alpha}}& \cdots & \displaystyle{\widehat{[x_1; y_j]_{\alpha}}}  &\cdots & \displaystyle{[x_1; y_n]_{\alpha}}  \\ 
\displaystyle{[x_2; y_1]_{\alpha}}& \cdots   & \displaystyle{\widehat{[x_2; y_j]_{\alpha}}}  & \cdots & \displaystyle{[x_2; y_n]_{\alpha}}  \\
\vdots& & \vdots& \ddots& \vdots \\
\displaystyle{[x_n; y_1]_{\alpha}}& \cdots  & \displaystyle{\widehat{[x_n; y_j]_{\alpha}}}  & \cdots & \displaystyle{[x_n; y_n]_{\alpha}}  
\end{array}
\right|   \\
& \qquad \qquad \qquad \qquad \qquad \qquad= \left|
\begin{array}{cccc}
1 & 1 & \cdots & 1 \\
\displaystyle{[x_1; y_1]_{\alpha}} & \displaystyle{[x_1; y_2]_{\alpha}}  &\cdots & \displaystyle{[x_1; y_{n+1}]_{\alpha}}  \\ 
\displaystyle{[x_2; y_1]_{\alpha}}  & \displaystyle{[x_2; y_2]_{\alpha}}  & \cdots & \displaystyle{[x_2; y_{n+1}]_{\alpha}}  \\
\vdots& \vdots& \ddots& \vdots \\
\displaystyle{[x_n; y_1]_{\alpha}}  & \displaystyle{[x_n; y_2]_{\alpha}}  & \cdots & \displaystyle{[x_n; y_{n+1}]_{\alpha}}  
\end{array}
\right| 
\end{split}
\end{equation}
The  left-hand-side of \eqref{bigidentity} can be rewritten as
\begin{equation}
\left| \begin{array}{cccc}
C_1 & C_2 & \cdots & C_{n+1} \\
\displaystyle{[x_1; y_1]_{\alpha}} & \displaystyle{[x_1; y_2]_{\alpha}}  &\cdots & \displaystyle{[x_1; y_{n+1}]_{\alpha}}  \\ 
\displaystyle{[x_2; y_1]_{\alpha}}  & \displaystyle{[x_2; y_2]_{\alpha}}  & \cdots & \displaystyle{[x_2; y_{n+1}]_{\alpha}}  \\
\vdots& \vdots& \ddots& \vdots \\
\displaystyle{[x_n; y_1]_{\alpha}}  & \displaystyle{[x_n; y_2]_{\alpha}}  & \cdots & \displaystyle{[x_n; y_{n+1}]_{\alpha}}  
\end{array}
\right|
\end{equation}
where, by  Lemma~\ref{lemiden},
\begin{equation}
C_j = 1 - \sum_{i=1}^n [x_i;y_j]_{\alpha}\left( \frac{x_i+1-\alpha}{\alpha} \right) \left( \prod_{k \neq i} \frac{x_k-x_i-1}{x_k-x_i}  \right)
\end{equation}  
Observing that
$$
f_i(x) = \left( \frac{x_i+1-\alpha}{\alpha} \right) \left( \prod_{k \neq i} \frac{x_k-x_i-1}{x_k-x_i}  \right)
$$
does not depend on $j$, we have that
\begin{equation}
\left| \begin{array}{cccc}
C_1 & C_2 & \cdots & C_{n+1} \\
\displaystyle{[x_1; y_1]_{\alpha}} & \displaystyle{[x_1; y_2]_{\alpha}}  &\cdots & \displaystyle{[x_1; y_{n+1}]_{\alpha}}  \\ 
\displaystyle{[x_2; y_1]_{\alpha}}  & \displaystyle{[x_2; y_2]_{\alpha}}  & \cdots & \displaystyle{[x_2; y_{n+1}]_{\alpha}}  \\
\vdots& \vdots& \ddots& \vdots \\
\displaystyle{[x_n; y_1]_{\alpha}}  & \displaystyle{[x_n; y_2]_{\alpha}}  & \cdots & \displaystyle{[x_n; y_{n+1}]_{\alpha}}  
\end{array} \right|
= \left|    \begin{array}{cccc}
1 & 1 & \cdots & 1 \\
\displaystyle{[x_1; y_1]_{\alpha}} & \displaystyle{[x_1; y_2]_{\alpha}}  &\cdots & \displaystyle{[x_1; y_{n+1}]_{\alpha}}  \\ 
\displaystyle{[x_2; y_1]_{\alpha}}  & \displaystyle{[x_2; y_2]_{\alpha}}  & \cdots & \displaystyle{[x_2; y_{n+1}]_{\alpha}}  \\
\vdots& \vdots& \ddots& \vdots \\
\displaystyle{[x_n; y_1]_{\alpha}}  & \displaystyle{[x_n; y_2]_{\alpha}}  & \cdots & \displaystyle{[x_n; y_{n+1}]_{\alpha}}  
\end{array}
\right| 
\end{equation}
since the first row of the matrix to the right can be obtained by doing the transformation $R_1 \to R_1 + f(x_1)R_2 +\cdots + f(x_n) R_{n+1}$ in the matrix to the left.

\noindent{\it Remaining terms (general case).}
In the general case, there are labels $x$ and $y$ in the same row or the same column.  We will denote
by ${\rm Rel}_{\Omega/\Lambda}$ the set of relations between those variables.  For instance, using
\begin{equation}
\tableau[scY]{ &&& \fl & \bl \,\, x_{1} & \bl \, y_{1}\\ &&& \fl & \bl  \, \, x_{2} &  \bl \, y_{2} \\ &&& \fl & \bl \, \, x_{3} & \bl \, y_{3} \\   \tf \ocircle &  \bl \cerclep &  \bl \, \, x_{4} & \bl  \, y_{4}}
\end{equation}
we find the relations
$$
{\rm Rel}_{\Omega/\Lambda}= \{ x_1=y_1, x_2=y_2, x_3=y_3, x_4=y_4-\alpha,x_1=x_2+1,x_2=x_3+1\}
$$
We need to  show that in the general case the remaining terms also give
$(-1)^{\#(\Omega/\Lambda)}{\rm Det_{\Omega/\Lambda}}$. 
 This will be done by showing that
the new identity induced by the remaining terms is  simply a limiting case of 
the identity obtained in \eqref{bigidentity}.  To be more precise, supposing that
\eqref{bigidentity} corresponds to $\mathcal A= \mathcal B$, where
$\mathcal A$ and $\mathcal B$ are respectively the l.h.s and r.h.s of \eqref{bigidentity}, the new identity
will be
\begin{equation} \label{newbigidentity}
\lim_{{\rm Rel}_{\Omega/\Lambda}} \prod (y_{j_r}-x_{i_r}-\alpha) \prod (x_{I_\ell}-y_{J_\ell}) \, \mathcal A =
\lim_{{\rm Rel}_{\Omega/\Lambda}} \prod (y_{j_r}-x_{i_r}-\alpha) \prod (x_{I_\ell}-y_{J_\ell}) \, \mathcal B 
\end{equation}
where the pairs $(x_{i_r}, y_{j_r})$ are those such that $x_{i_r}= y_{j_r}-\alpha$ (that is, those such that $x_{i_r}$ labels
a cell just to the left of the cell labeled by $y_{j_r}$)
while the pairs $(x_{I_\ell}, y_{J_\ell})$
are those such that  $x_{I_\ell}= y_{J_\ell}$ (that is, those such that
$x_{I_\ell}$ and  $y_{J_\ell}$ label  cells in the same position).  The fact  that $\mathcal A= \mathcal B$ will then
readily imply \eqref{newbigidentity}.  From Remark~\ref{remark}, we have that
$$
\lim_{{\rm Rel}_{\Omega/\Lambda}} \prod (y_{j_r} -x_{i_r}-\alpha) \prod (x_{I_\ell}-y_{J_\ell}) \, \mathcal B =
(-1)^{\#(\Omega/\Lambda)} {\rm Det_{\Omega/\Lambda}}
$$
where
$$
\mathcal B =
(-1)^{\#(\Omega/\Lambda)}
\left|
\begin{array}{cccc}
1 & 1 & \cdots & 1 \\
\displaystyle{[x_1; y_1]'_{\alpha}} & \displaystyle{[x_1; y_2]'_{\alpha}}  &\cdots & \displaystyle{[x_1; y_{n+1}]'_{\alpha}}  \\ 
\displaystyle{[x_2; y_1]'_{\alpha}}  & \displaystyle{[x_2; y_2]'_{\alpha}}  & \cdots & \displaystyle{[x_2; y_{n+1}]'_{\alpha}}  \\
\vdots& \vdots& \ddots& \vdots \\
\displaystyle{[x_n; y_1]'_{\alpha}}  & \displaystyle{[x_n; y_2]'_{\alpha}}  & \cdots & \displaystyle{[x_n; y_{n+1}]'_{\alpha}}  
\end{array}
\right|  =: (-1)^{\#(\Omega/\Lambda)}  {\rm Det_{\Omega/\Lambda}'}
$$
with
\begin{equation}
[x_i;y_j]'_\alpha = \frac{\alpha}{(x_i-y_j+\alpha)(x_i-y_j)} 
\end{equation}

We now show that the contribution of the remaining terms is
$$
\lim_{{\rm Rel}_{\Omega/\Lambda}} \prod (y_{j_r}-x_{i_r}-\alpha) \prod (x_{I_\ell}-y_{J_\ell}) \, \mathcal A
$$
For this, it suffices to show that the contribution of $y_j$ is equal to
\begin{equation} \label{geneq}
\begin{split}
&  (-1)^{\#(\Omega/\Lambda)+j-1} \lim_{{\rm Rel}_{\Omega/\Lambda}} \prod (y_{j_r}-x_{i_r}-\alpha) \prod (x_{I_\ell}-y_{J_\ell})  \, \times \\
& \qquad  
  \Biggl[ \left(1-  \frac{y_j+1-\alpha}{\alpha}\right)
    \left( \prod_{i=1}^n \frac{x_i-y_j+\alpha-1}{x_i-y_j+\alpha} \right)  
+ 
\left(\frac{y_j+1-\alpha}{\alpha}\right) \left( \prod_{i=1}^n \frac{x_i-y_j-1}{x_i-y_j} \right) \Biggr] {\rm Det}'_{\Omega/\Lambda^{(j)}}  
\end{split}
\end{equation}
where, as before, ${\rm Det}'_{\Omega/\Lambda^{(j)}}$ is equal to ${\rm Det}_{\Omega/\Lambda^{(j)}}$
with $[x_i;y_\ell]_\alpha$ replaced by  $[x_i;y_\ell]'_\alpha$.  In order to show \eqref{geneq}, we will show that the extra contribution
stemming from $(\tilde e_0 -Q) e_n$ is
\begin{equation} \label{geneq1}
 (-1)^{\#(\Omega/\Lambda)+j-1} \! \!\lim_{{\rm Rel}_{\Omega/\Lambda}} \prod (y_{j_r}-x_{i_r}-\alpha) \prod (x_{I_\ell}-y_{J_\ell}) 
 \Biggl[ \left(1-  \frac{y_j+1-\alpha}{\alpha}\right) \left( \prod_{i=1}^n \frac{x_i-y_j+\alpha-1}{x_i-y_j+\alpha} \right)\Biggr] {\rm Det}'_{\Omega^{(j)}/\Lambda}  
\end{equation}
(note that we used the fact that, by definition, ${\rm Det}'_{\Omega^{(j)}/\Lambda}= {\rm Det}'_{\Omega/\Lambda^{(j)}}$)
while that stemming from $e_n Q$ is
\begin{equation} \label{geneq2}
 (-1)^{\#(\Omega/\Lambda)+j-1} \! \!
  \lim_{{\rm Rel}_{\Omega/\Lambda}} \prod (y_{j_r}-x_{i_r}-\alpha) \prod (x_{I_\ell}-y_{J_\ell}) 
 \Biggl[ 
\left(\frac{y_j+1-\alpha}{\alpha}\right) \left( \prod_{i=1}^n \frac{x_i-y_j-1}{x_i-y_j} \right) \Biggr] {\rm Det}'_{\Omega/\Lambda^{(j)}}  
\end{equation}

There are many cases to consider.  Suppose first that $y_j$ labels a circle and 
that there is no label $x$ in the column and the row of $y_j$.  From the action of
$(\tilde e_0 -\tilde Q) e_n$,  we get by {\bf (C1)} that for all $s$ above $y_j$ and in the row of some
$x_i$, we have an extra contribution of
$$
h^A_{\Omega/\Omega^{(j)}}(s)= \frac{x_i-y_j+\alpha-1}{x_i-y_j+\alpha}
$$
Note that a circle (with a $y$ label) to the right of the cell corresponding to $x_i$ does not change the result since
$h^A_{\Omega/\Omega^{(j)}}(s)$ is not affected by a circle at the end of the row.

By {\bf (C3)} this time, for all  $s$ in the row of  $y_j$ and in the column of some
$x_i$, we have an extra contribution
$$
h^B_{\Omega^{(j)}/\Lambda}(s)= \prod_i \frac{y_j-x_i-\alpha+1}{y_j-x_i-\alpha}= \prod_i \frac{x_i-y_j+\alpha-1}{x_i-y_j+\alpha}
$$
where the product is over all $x_i$'s in that column.  Note that the previous equality holds by Lemma~\ref{lemma1}.
Also note that if there is a circle (with a $y$ label) at the end of the column, it does not contribute since $h^B_{\Omega^{(j)}/\Lambda}(s)$ does not
take that circle into account.

The remaining contribution from $\Omega^{(j)}$ in $(\tilde{e}_{0}-\tilde Q)e_{n}$ is thus as in \eqref{termgen1}:
\begin{equation} \label{eq424}
(-1)^{\#(\Omega^{(j)}/\Lambda)+\#(\Omega/\Omega^{(j)})}\left(1-\frac{y_j+1-\alpha}{\alpha}\right) \left( \prod_{i=1}^n \frac{x_i-y_j+\alpha-1}{x_i-y_j+\alpha} \right) {\rm Det}_{\Omega^{(j)}/\Lambda} 
\end{equation}
We have in this case that (compare to the equations leading to \eqref{eqsign1} in which there were no new squares)
$$
(-1)^{\#(\Omega/\Lambda)} = (-1)^{\#(\Omega^{(j)}/\Lambda)+\# (\text{preexisting circles above } y_j)+ \# (\text{new squares above } y_j)} 
$$
and 
$$
(-1)^{\#(\Omega/\Omega^{(j)})} = (-1)^{\# (\text{preexisting circles above } y_j) + \# (\text{new circles above } y_j )} 
$$
This gives again that $(-1)^{\#(\Omega^{(j)}/\Lambda)+\#(\Omega/\Omega^{(j)})}=  (-1)^{\#(\Omega/\Lambda)+j-1}$ since the total number of new circles and new squares above $y_j$ is $j-1$.
The contribution \eqref{eq424} from $\Omega^{(j)}$ in $(\tilde{e}_{0}-\tilde Q)e_{n}$ is thus equal to \eqref{geneq1}
since
\begin{equation}
  \lim_{{\rm Rel}_{\Omega/\Lambda}} \prod (y_{j_r}-x_{i_r}-\alpha) \prod (x_{I_\ell}-y_{J_\ell}) \,
{\rm Det}'_{\Omega^{(j)}/\Lambda} = {\rm Det}_{\Omega^{(j)}/\Lambda}
\end{equation}

The remaining terms
stemming from the action of $e_n \tilde Q$ will lead to \eqref{geneq2} in a very similar way.  
By {\bf (C2)}, for all $s$ above $y_j$ and in the row of some
$x_i$, we have an extra contribution of
$$
h_{\Lambda^{(j)}/\Lambda}^A(s)= \frac{x_i-y_j-1}{x_i-y_j}
$$
while by {\bf (C4)}, for all  $s$ in the row of  $y_j$ and in the column of some
$x_i$, we have by Lemma~\ref{lemma1} an extra contribution of
$$
h^B_{\Omega/\Lambda^{(j)}}(s)= \prod_i \frac{y_j-x_i+1}{y_j-x_i}= \prod_i \frac{x_i-y_j-1}{x_i-y_j}
$$
where the product is over all $x_i$'s in that column.  The contribution of $e_n \tilde Q$
is thus exactly as in \eqref{termgen2}:
$$
(-1)^{\#(\Lambda^{(j)}/\Lambda)+\#(\Omega/\Lambda^{(j)})}\left(\frac{y_j+1-\alpha}{\alpha}\right) \left( \prod_{i=1}^n \frac{x_i-y_j-1}{x_i-y_j} \right) {\rm Det}_{\Omega/\Lambda^{(j)}} 
$$
This leads to
\eqref{geneq2} just as the previous extra contribution of $(\tilde e_0 -\tilde Q)e_n$ led to
\eqref{geneq1}.  We stress that the sign is the correct one since in this case we have
(compare this time to the equations leading to \eqref{eqsign1p} in which there were no new squares)
\begin{equation} \nonumber
  \begin{split}
    & (-1)^{\#(\Omega/\Lambda)-\#(\Omega/\Lambda^{(j)})} = \\
    & (-1)^{\# (\text{new circles below } y_j)-\# (\text{bumping squares below } y_j)+\# (\text{preexisting circles above } y_j)+\# (\text{new squares above } y_j)}
   \end{split}
\end{equation}
and 
$$
(-1)^{\#(\Lambda^{(j)}/\Lambda)} = (-1)^{\# (\text{preexisting circles above } y_j) +\# (\text{bumping squares above } y_j) } 
$$
Using the fact that the total number of new circles and bumping squares above or below $y_j$ is even
($e_n$ contains the same number of bumping squares and new circles), we conclude as wanted that
\begin{equation} \label{eqsign2}
    (-1)^{\#(\Omega/\Lambda^{(j)})+\#(\Lambda^{(j)}/\Lambda)}  =  (-1)^{\#(\Omega/\Lambda)+\# (\text{new circles above } y_j)+\# (\text{new squares above } y_j)}= (-1)^{\#(\Omega/\Lambda)+j-1}
\end{equation}

Now suppose that there is a bumping square and/or new squares above $y_j$ in its column (but not in its row).
For the action of $(\tilde e_0 -\tilde Q) e_n$,  
the
analysis is exactly as before and we get that \eqref{geneq1} holds.  In the action of $e_n \tilde Q $ the contribution is zero since
$\tilde Q$ cannot put a new circle in the position corresponding to $y_j$.  We thus have to show that in that case,
\begin{equation} \label{eqzero}
\lim_{{\rm Rel}_{\Omega/\Lambda}} \prod (y_{j_r}-x_{i_r}-\alpha) \prod (x_{I_\ell}-y_{J_\ell}) 
 \Biggl[ 
   \left(\frac{y_j+1-\alpha}{\alpha}\right) \left( \prod_{i=1}^n \frac{x_i-y_j-1}{x_i-y_j} \right) \Biggr] {\rm Det}'_{\Omega/\Lambda^{(j)}}
 = 0
\end{equation}
We have that the cell above that corresponding to $y_\ell$ has a label $x_a$ such that $x_a=y_\ell+1$ (which belongs to the
set of relations ${\rm Rel}_{\Omega/\Lambda}$).  The claim then holds due to the presence of the factor $x_a-y_\ell-1$ in the numerator of
  the l.h.s. of
\eqref{eqzero}.

Now suppose that there is a bumping square (labeled by $x_a$) in the row of $y_j$.  For the action of $(\tilde e_0 -\tilde Q) e_n$,  
the
analysis is exactly as before except that the bumping square in the row of $y_j$ contributes 1 instead of
$(x_a-y_j+\alpha-1)/(x_a-y_j+\alpha)$.  But since $x_a=y_j-\alpha$ is a relation, we have 
$$
\lim_{x_a=y_j-\alpha} (y_j-x_a-\alpha)  \frac{(x_a-y_j+\alpha-1)}{(x_a-y_j+\alpha)} =1
$$
which implies that \eqref{geneq1} still holds.   In the action of $e_n \tilde Q $ the contribution is zero since
$\tilde Q$ cannot put a new circle in the position corresponding to $y_j$.  But this is consistent with 
\begin{equation}
\lim_{{\rm Rel}_{\Omega/\Lambda}} \prod (y_{j_r}-x_{i_r}-\alpha) \prod (x_{I_\ell}-y_{J_\ell}) 
 \Biggl[ 
   \left(\frac{y_j+1-\alpha}{\alpha}\right) \left( \prod_{i=1}^n \frac{x_i-y_j-1}{x_i-y_j} \right) \Biggr] {\rm Det}'_{\Omega/\Lambda^{(j)}}
 = 0
\end{equation}
since the relation $x_a=y_j-\alpha$ in  ${\rm Rel}_{\Omega/\Lambda}$ implies that there
will be a factor $y_j-x_a-\alpha=0$ in the numerator.

Suppose that $y_j$ corresponds to a new square (also labeled by $x_a$) without new squares or a bumping square above it in its column.
In the action of $e_n \tilde Q $,
all the cells above $y_j$ contribute as if $y_j$ were a circle since the contribution $h^A_{\Lambda^{(j)}/\Lambda}$ is unchanged.
Similarly, all the cells in the row of $y_j$ contribute as if $y_j$ were a circle since $h^B_{\Omega/\Lambda^{(j)}}$ does not distinguish
between a circle and a square in its arm.  Since the cell corresponding to $y_j$ only contributes
$(y_j+\alpha-1)/\alpha$, the factor $(x_a-y_j-1)/(x_a-y_j)$ is missing in the product.
But given that $x_a=y_j$ is a relation, we have 
$$
\lim_{x_a=y_j} (x_a-y_j)  \frac{(x_a-y_j-1)}{(x_a-y_j)} =-1 
$$
This extra sign is compensated by the fact that in the analysis of the sign leading to \eqref{eqsign2}, 
the total number of new circles and bumping squares above or below $y_j$ is now odd (the bumping square coming from the action
of $e_n$ in the row of $y_j$ is not counted since it is neither above or below $y_j$), which leads to   
$$
 (-1)^{\#(\Omega/\Lambda^{(j)})+\#(\Lambda^{(j)}/\Lambda)} = (-1)^{\#(\Omega/\Lambda)+j}
$$
Hence,  \eqref{geneq1} still holds. 

In the action of $(\tilde e_0-\tilde Q) e_n$ the contribution is zero since
$\tilde e_0-\tilde Q$ cannot add a new circle in the position corresponding to $y_j$.  Again, this is consistent with
\begin{equation} \label{eqzero0}
\lim_{{\rm Rel}_{\Omega/\Lambda}} \prod (y_{j_r}-x_{i_r}-\alpha) \prod (x_{I_\ell}-y_{J_\ell}) 
 \Biggl[ \left(1-  \frac{y_j+1-\alpha}{\alpha}\right) \left( \prod_{i=1}^n \frac{x_i-y_j+\alpha-1}{x_i-y_j+\alpha} \right)\Biggr] {\rm Det}'_{\Omega^{(j)}/\Lambda} =0 
\end{equation}
since  the relation $x_a=y_j$ in ${\rm Rel}_{\Omega/\Lambda}$ implies that there will be a factor
$x_a-y_j=0$ in the numerator.

Suppose finally that $y_j$ corresponds to a new square (also labeled by $x_a$) with new squares or a bumping square above it in its column.
In this case  both in the action of $e_n \tilde Q $ and $(\tilde e_0 -\tilde Q)e_n$
the contribution is zero since neither
$\tilde Q$ or $\tilde e_0 -\tilde Q$ can put a new circle in the position corresponding to $y_j$.  Again, as in \eqref{eqzero} and
\eqref{eqzero0}, this is consistent with
\begin{equation}
\lim_{{\rm Rel}_{\Omega/\Lambda}} \prod (y_{j_r}-x_{i_r}-\alpha) \prod (x_{I_\ell}-y_{J_\ell}) 
 \Biggl[ 
   \left(\frac{y_j+1-\alpha}{\alpha}\right) \left( \prod_{i=1}^n \frac{x_i-y_j-1}{x_i-y_j} \right) \Biggr] {\rm Det}'_{\Omega/\Lambda^{(j)}}
 = 0
\end{equation}
and
\begin{equation}
  \lim_{{\rm Rel}_{\Omega/\Lambda}} \prod (y_{j_r}-x_{i_r}-\alpha) \prod (x_{I_\ell}-y_{J_\ell}) 
 \Biggl[ \left(1-  \frac{y_j+1-\alpha}{\alpha}\right) \left( \prod_{i=1}^n \frac{x_i-y_j+\alpha-1}{x_i-y_j+\alpha} \right)\Biggr] {\rm Det}'_{\Omega^{(j)}/\Lambda} =0 
\end{equation}

\subsection{Second proposition}
As described in the proof of Theorem~\ref{theo}, in order to finish our proof of the Pieri rule for Jack polynomials in superspace, we also need to prove
the induction step for the Pieri rule in the case of $e_n$. 
\begin{proposition} \label{mainprop2}
If the Pieri rules hold for $\tilde e_{n-1}$, then they also hold for
$e_n$.
\end{proposition}
The proof of the proposition  can be done by following the steps of the proof of Proposition~\ref{mainprop}.  Since this is a very
straightforward but tedious task, we will omit its presentation in this article.

\section{Conjectured Pieri rules for the Macdonald polynomials in superspace} \label{secMac}

The Pieri rules for the Macdonald polynomials in superspace appear to be very similar to the ones for the Jack polynomials in superspace.
Before stating the conjecture, we need to introduce the quantity $d(\Omega/\Lambda)$.  Read the labels $x$ and $y$ from top to bottom to form a word $w$ in the letters $x_i$ and $y_j$ (if $x_\ell$ and $y_r$ are in the same row, then
$y_r$ is read first).  Then let
\begin{equation}
d(\Omega/\Lambda)  = d_1(\Omega/\Lambda) + d_2(\Omega/\Lambda)  + \cdots + d_n(\Omega/\Lambda) 
\end{equation}
where
$$
d_i(\Omega/\Lambda) = \# \{ y_j \text{ to the right of } x_i \, |\,  j\leq i  \} -  \# \{ y_j \text{ to the left of } x_i \, |\,  j>i  \} 
$$
In some sense,  $d(\Omega/\Lambda)$ measures the distance between $w$ and the words $w_0=y_1 x_1 y_2 x_2 \cdots y_n x_n$ or $w_0=y_1 x_1 y_2 x_2 \cdots y_n x_n y_{n+1}$ (depending on whether $\Omega/\Lambda$ is a vertical $n$-strip or a vertical $\tilde n$-strip),
where elementary transpositions of the type $(x,y) \to (y,x)$ count as -1 and those of the type $(y,x) \to (x,y)$ count as +1.  For instance, if the strip
$\Omega/\Lambda$ is given by
$${\scriptsize \tableau[scY]{ \bl  &\bl& \bl  &\bl &\bl \bl  &\bl \tf \circ \\  \bl  &\bl &\bl  &\bl \bl    &\bl \cerclep \\  \bl  &\bl \bl  &\bl \bl  &\bl \cerclep \\  \bl  &\bl \bl  &\bl \tf \circ \\   \bl &\bl \tf \circ \\     \bl   \cerclep}}$$
then $w=x_1 y_1 y_2 x_2 x_3 y_3$, which gives $d_1=1$, $d_2=0$ and $d_3=1$ for a total value of $d(\Omega/\Lambda)=2$.

\begin{conjecture} \label{conj}
  The Pieri rules for the Macdonald polynomials in superspace are given by
\begin{equation} \label{pierimac}
e_n \, P_\Lambda^{(q,t)} = \sum_\Omega v_{\Lambda \Omega}(q,t) \, P_\Omega^{(q,t)} 
\qquad {\rm and} \qquad \tilde e_n \, P_\Lambda^{(q,t)} = \sum_\Omega \tilde v_{\Lambda \Omega}(q,t) \, P_\Omega^{(q,t)} 
\end{equation}
where the sum is over all $\Omega$'s such that $\Omega/\Lambda$ is a vertical $n$-strip and a vertical  $\tilde n$-strip
respectively.  Moreover,
\begin{equation} \label{eqPieriMac}
v_{\Lambda \Omega}(q,t)= (-1)^{\#(\Omega/\Lambda)}  t^{d(\Omega/\Lambda)} 
\psi_{\Omega/\Lambda}' {\rm Det_{\Omega/\Lambda}}  \qquad {\rm and} \qquad
 \tilde v_{\Lambda \Omega}(q,t)=  (-1)^{\#(\Omega/\Lambda)} t^{d(\Omega/\Lambda)}  \psi_{\Omega/\Lambda}' {\rm Det_{\Omega/\Lambda}} 
\end{equation}
where $\psi_{\Omega/\Lambda}'$ and  ${\rm Det_{\Omega/\Lambda}}$ are as described in Section~\ref{SecPieri}, but with 
$h_\Lambda^{(\alpha)}(s)$ and $h^\Lambda_{(\alpha)}(s)$ replaced respectively by
\begin{equation}
h_\Lambda^{(q,t)}(s)= 1-t^{\ell_{\Lambda^\circledast}(s)}q^{a_{\Lambda^*}(s)+1} \qquad {\rm and} \qquad
h^\Lambda_{(q,t)}(s)= 1-t^{\ell_{\Lambda^*}(s) + 1}q^{a_{\Lambda^\circledast}(s)} 
\end{equation}
and $[x;y]_{(\alpha)}$ replaced by
\begin{equation} \label{xyqt}
[x;y]_{(q,t)} = \frac{(t-1)(q^{1/2}-q^{-1/2})}{(t^{(x-y)/2}-t^{-(x-y)/2})(q^{1/2}t^{(x-y)/2}-q^{-1/2}t^{-(x-y)/2})}
\end{equation}
\end{conjecture}  
It is easy to see that Conjecture~\ref{conj} becomes Theorem~\ref{theo}
in the Jack limit $q=t^\alpha, t\to 1$.  The main reason we cannot prove  Conjecture~\ref{conj}
by simply extending the proof of Theorem~\ref{theo} is that
$q$-analogs of the operators $q^\perp$ and
$Q$ need to be used (the operators $q^\perp$ and $Q$  do not have a simple 
action on Macdonald polynomials in superspace).  We believe that such $q$-analogs belong to an
extension of the super-Poincar\'e algebra defined in 
\cite{BDM}.  Another drawback is that the explicit action of  $\tilde e_0$ on the Macdonald polynomials in superspace, as well as (obviously)
that of the still unknown $q$-analogs of $q^\perp$ and $Q$, has not yet been  established.   In order to prove
their explicit action, it is thus first necessary to 
prove the formulas for the norm and the evaluation of the Macdonald polynomials in superspace that were
conjectured in \cite{BDLM1} (these formulas involve products of $h_{(q,t)}^\Lambda$ and $h^{(q,t)}_\Lambda$ 
just as in the non-supersymmetric case).  This problem is being considered in \cite{GL}.

Another approach to prove Conjecture~\ref{conj} would be to generalize the proof of the Pieri rules for Macdonald polynomials that can be found in \cite{Mac} (and which is due to Koornwinder \cite{K}).
This proof relies on the explicit action of the Macdonald operators $D_n^r$ to prove at the same time the symmetry and the Pieri rules
for the Macdonald polynomials.  Although there is a conjectured version of the symmetry for Macdonald polynomials in superspace \cite{B},
it does not seem straightforward to generalize this approach given that the analogs of the Macdonald operators in superspace are not known explicitly (they are only known as expressions involving Cherednik operators \cite{BDLM2}).

In the limits $q=t=0$ and $q=t=\infty$, the Macdonald polynomials in superspace $P_\Lambda^{(q,t)}$ become the Schur functions in superspace $s_\Lambda$
and $\bar s_\Lambda$ respectively.  The corresponding Pieri rules for $s_\Lambda$
and $\bar s_\Lambda$, which surprisingly have coefficients always equal to $1$, $-1$ or 0,
have been established in  \cite{BM,JL}.  One would thus expect to be easily
able to
give extra support for Conjecture~\ref{conj} by recovering the Schur Pieri rules
in the specials cases $q=t=0$ and $q=t=\infty$.
Unfortunately, since obtaining those limiting cases proved not to be so trivial, we decided not to include them in this article.

\section{Connection with the 6-vertex model} \label{sec6vertex}

Quite remarkably, 
the determinant ${\rm Det}_{\Omega/\Lambda}$ appearing in the Pieri rules for the Macdonald polynomials in the case
where $\Omega/\Lambda$ is a vertical $n$-strip is, as we will see,
essentially the Izergin-Korepin determinant related
to the partition function of the 6-vertex model \cite{Bre,Ize,Ku}, probably the most fundamental exactly solvable model in
statistical mechanics.  This model is also known as square ice given that it can be interpreted as molecules of water 
(${\rm H}_2{\rm O}$) on a two-dimensional square lattice 
(see Figure~\ref{fig1}, where the 6 configurations are the 2 straight ones, vertical and horizontal, and the 4 possible orientations of the bent molecule).
\begin{figure}
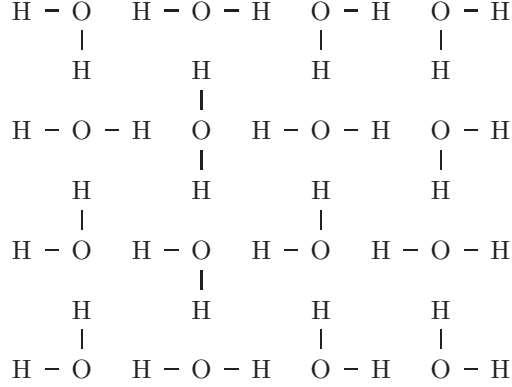

\caption{A square ice configuration}
\label{fig1}
{
$$
\dgARROWLENGTH=0.5em
\begin{diagram}
\node {\rm H } \node{ \rm O}\arrow{w,-}{} \arrow{s,-}{} \node{ \rm  H}  \node{ \rm  O}\arrow{w,-}{}\arrow{e,-}{} \node{ \rm  H} \node{ \rm  O} \arrow{s,-}{} \arrow{e,-}{}\node{ \rm  H} \node{ \rm  O} \arrow{s,-}{} \arrow{e,-}{}\node{ \rm  H} \\
\node {} \node{ \rm H } \node{ \rm  }  \node{ \rm  H} \node{ \rm  } \node{ \rm  H}  \node{ \rm  } \node{ \rm  H}  \node{ \rm  } \\
\node { \rm H } \node{ \rm  O}\arrow{w,-}{} \arrow{e,-}{} \node{ \rm  H}  \node{ \rm  O}\arrow{n,-}{}\arrow{s,-}{} \node{ \rm  H} \node{ \rm  O} \arrow{w,-}{} \arrow{e,-}{}\node{ \rm  H} \node{ \rm  O} \arrow{s,-}{} \arrow{e,-}{}\node{ \rm  H} \\
\node {} \node{ \rm H } \node{ \rm  }  \node{ \rm  H} \node{ \rm  } \node{ \rm  H}  \node{ \rm  } \node{ \rm  H}  \node{ \rm  } \\
\node { \rm H } \node{ \rm  O}\arrow{n,-}{} \arrow{w,-}{} \node{ \rm  H}  \node{ \rm  O}\arrow{w,-}{}\arrow{s,-}{} \node{ \rm  H} \node{ \rm  O} \arrow{w,-}{} \arrow{n,-}{}\node{ \rm  H} \node{ \rm  O} \arrow{w,-}{} \arrow{e,-}{}\node{ \rm  H} \\
\node {} \node{ \rm H } \node{ \rm  }  \node{ \rm  H} \node{ \rm  } \node{ \rm  H}  \node{ \rm  } \node{ \rm  H}  \node{ \rm  }  \\
\node { \rm H } \node{ \rm  O}\arrow{n,-}{} \arrow{w,-}{} \node{ \rm  H}  \node{ \rm  O}\arrow{w,-}{}\arrow{e,-}{} \node{ \rm  H} \node{ \rm  O} \arrow{e,-}{} \arrow{n,-}{}\node{ \rm  H} \node{ \rm  O} \arrow{e,-}{} \arrow{n,-}{}\node{ \rm  H} 
\end{diagram}
$$}
\end{figure}
When the domain wall boundary conditions are imposed
($H$ atoms along the left and right edges and no $H$ atom along the top and bottom edges), the square ice configurations are in correspondence with
alternating sign matrices \cite{Bre,Ku}.  These matrices are square matrices whose row and column sums are equal to 1 and 
whose non-zero entries in each row and column alternate between 1 and -1.
For instance the  alternating sign matrix corresponding to the configuration in 
Figure~\ref{fig1} is given by 
$$
\left(\begin{array}{ccccc}
0 & {1} & 0 & 0 & 0\\
{1}& {-1} & 0 & {1}  & 0 \\
0 & 0 & {1} & {-1} & {1}\\
0 &  {1}& 0 & 0 & 0\\
0 & 0 & 0 & {1} & 0
\end{array}
\right)
$$  
where horizontal (resp. vertical) water molecules are now 1's (resp. -1's) while the bent water molecules are now 0's.

The Izergin-Korepin determinant related to the partition function of the 6-vertex model with domain wall boundary conditions is
 \cite{Bre,Ize,Ku}
\begin{equation}
  D( \mathbf x,\mathbf y; a)= \det \left(  \frac{(a-a^{-1})^2}{\Bigl(\mathbf x_i/\mathbf y_j- (\mathbf x_i/\mathbf y_j)^{-1}\Bigr)\Bigl(a\mathbf x_i/\mathbf y_j-
(a\mathbf x_i/\mathbf y_j)^{-1}  \Bigr)} \right)_{1\leq i,j \leq n}
\end{equation}
For convenience, let
\begin{equation} \label{eqdp}
D'( \mathbf x,\mathbf y; a)=   \left(\frac{t-1}{q^{1/2}-q^{-1/2}}\right)^n D( \mathbf x,\mathbf y; a)
\end{equation}
When $\Omega/\Lambda$ is a vertical $n$-strip such that
the $x$ and $y$ labels are never in the same row, we have immediately
from \eqref{xyqt} that
\begin{equation}
{\rm Det}_{\Omega/\Lambda}= D'( \mathbf x,\mathbf y; a)  \quad \text{if }\, \, \mathbf x_i = t^{x_i/2}, \mathbf y_j = t^{y_j/2} \text{ and } a=q^{1/2}
\end{equation}
In the general case, if we extend Remark~\ref{remark} to the Macdonald case,  we have again that  ${\rm Det_{\Omega/\Lambda}}$
is related to the  Izergin-Korepin determinant.  Indeed,
it is straightforward to see that when  $\mathbf x_i = t^{x_i/2}$, $\mathbf y_j = t^{y_j/2}$ and $a=q^{1/2}$, we get that
\begin{equation}
 {\rm Det_{\Omega/\Lambda}} =  \lim_{\substack{x_{I_\ell}=y_{J_\ell}\\ x_{i_r}=y_{j_r}-\alpha}}  \prod \frac{(t^{(x_{I_\ell}-y_{J_\ell})/2}-t^{-(x_{I_\ell}-y_{J_\ell})/2})}{(t-1)}\prod \frac{(t^{-(x_{i_r}-y_{j_r}+\alpha)/2}-t^{(x_{i_r}-y_{j_r}+\alpha)/2})}{(t-1)}  \,
 D'( \mathbf x,\mathbf y; a) 
\end{equation}
where as before the pairs $(x_{i_r}, y_{j_r})$ are those such that $x_{i_r}= y_{j_r}-\alpha$ (recall that $q=t^\alpha$),
while the pairs $(x_{I_\ell}, y_{J_\ell})$
are those such that  $x_{I_\ell}= y_{J_\ell}$.

From this connection, it is immediate 
that ${\rm Det_{\Omega/\Lambda}}$ can always be obtained as a weighted sum over alternating sign matrices in the following way \cite{Br}.  Let the weight of the entry in position $s=(i,j)$ in an alternating sign matrix be
\begin{itemize}
\item  $z_{ij}$ (resp. $z_{ij}^{-1}$) if the corresponding entry is  $1$ (resp. $-1$)  \\
\item $[az_{ij}]$ (resp. $[z_{ij}]$) if the corresponding entry is  0 and  
  the sum of the entries in the column above $s$ and in the row to the left of $s$ is even (resp. odd) 
\end{itemize}
where $z_{ij}=\mathbf x_i / \mathbf y_j$, and where
\begin{equation}
[z]= \frac{z-z^{-1}}{a-a^{-1}}
\end{equation}  
The weight $w(\mathcal A)$
of an alternating sign matrix $\mathcal A$ is defined to be the product of the weights of all entries in the matrix.  It is known that
\begin{equation}
  D( \mathbf x,\mathbf y; a) =
  \frac{\prod_{i<j} [\mathbf x_i / \mathbf x_j][\mathbf y_j / \mathbf y_i] }{\prod_{i=1} \mathbf x_i / \mathbf y_i \prod_{i,j} [\mathbf x_i / \mathbf y_j] [a\mathbf x_i / \mathbf y_j] }  \sum_{\mathcal A} w(\mathcal A) 
\end{equation}  
where the sum is over all alternating sign matrices of size $n$. From \eqref{eqdp}, we thus have in our case 
\begin{equation}
  {\rm Det}_{\Omega/\Lambda} =   \left(\frac{t-1}{q^{1/2}-q^{-1/2}}\right)^n 
  \frac{\prod_{i<j} [t^{(x_i-x_j)/2}][t^{(y_j-y_i)/2}] }{\prod_{i=1} t^{(x_i-y_i)/2} \prod_{i,j} [t^{(x_i-y_j)/2}] [q^{1/2}t^{(x_i-y_j)/2}] }  \sum_{\mathcal A} w(\mathcal A)
\end{equation}
where $[z]$ now stands for
\begin{equation}
[z]= \frac{z-z^{-1}}{q^{1/2}-q^{-1/2}}
\end{equation}
In the Jack limit  ($q=t^\alpha$ and $t\to 1$), this reduces to
\begin{equation} \label{detasm}
{\rm Det}_{\Omega/\Lambda} = {\alpha^{n^2}} 
  \frac{\prod_{i<j} (x_i-x_j)(y_j-y_i) }{\prod_{i,j} (x_i-y_j) (x_i-y_j+\alpha)}  \sum_{\mathcal A} w_\alpha(\mathcal A)
\end{equation}
where $w_\alpha(\mathcal A)$, the Jack limit of $w(\mathcal A)$, is such that 
the weight of the entry in position $s=(i,j)$ 
is 1 if the corresponding entry is  $1$ or $-1$,  and  $(x_i-y_j+\alpha)/\alpha$ (resp. $(x_i-y_j)/\alpha$) if the corresponding entry is  0 and  
  the sum of the entries in the column above $s$ and in the row to the left of $s$ is even (resp. odd). 
For example, the  7 alternating sign matrices of size 3 with their corresponding weight
are: 
 \begin{align} 
&  {\scriptsize \left(  \begin{array}{ccc}
  1 & 0 & 0 \\
  0 & 1 & 0 \\
  0 & 0 & 1
\end{array} \right)  \quad    (x_1-y_2) (x_1-y_3) (x_2-y_1) (x_2-y_3)(x_3-y_1) (x_3-y_2) \alpha^{-6}} \nonumber \\
& {\scriptsize \left(  \begin{array}{ccc}
  1 & 0 & 0 \\
  0 & 0 & 1 \\
  0 & 1 & 0
\end{array} \right)   \quad  (x_1-y_2) (x_1-y_3) (x_2-y_1) (x_2-y_2+\alpha)(x_3-y_1) (x_3-y_3+\alpha) \alpha^{-6}} \nonumber\\
& {\scriptsize \left(  \begin{array}{ccc}
  0 & 1 & 0 \\
  1 & 0 & 0 \\
  0 & 0 & 1
\end{array} \right)  \quad  (x_1-y_1+\alpha) (x_1-y_3) (x_2-y_2+\alpha) (x_2-y_3)(x_3-y_1) (x_3-y_2) \alpha^{-6}} \nonumber\\
   &
  {\scriptsize \left(  \begin{array}{ccc}
  0 & 1 & 0 \\
  0 & 0 & 1 \\
  1 & 0 & 0
\end{array} \right) \quad  (x_1-y_1+\alpha) (x_1-y_3) (x_2-y_1+\alpha) (x_2-y_2)(x_3-y_2+\alpha) (x_3-y_3+\alpha) \alpha^{-6}} \nonumber\\
& {\scriptsize \left(  \begin{array}{ccc}
  0 & 0 & 1 \\
  0 & 1 & 0 \\
  1 & 0 & 0
\end{array} \right) \quad  (x_1-y_1+\alpha) (x_1-y_2+\alpha) (x_2-y_1+\alpha) (x_2-y_3+\alpha)(x_3-y_2+\alpha) (x_3-y_3+\alpha) \alpha^{-6}} \nonumber\\
& {\scriptsize \left(  \begin{array}{ccc}
  0 & 0 & 1 \\
  1 & 0 & 0 \\
  0 & 1 & 0
\end{array} \right)   \quad  (x_1-y_1+\alpha) (x_1-y_2+\alpha) (x_2-y_2) (x_2-y_3+\alpha)(x_3-y_1) (x_3-y_3+\alpha) \alpha^{-6}} \nonumber\\
& {\scriptsize \left(  \begin{array}{ccc}
  0 & 1 & 0 \\
  1 & -1 & 1 \\
  0 & 1 & 0
\end{array} \right)  \quad  (x_1-y_1+\alpha) (x_1-y_3) (x_3-y_1) (x_3-y_3+\alpha) \alpha^{-4}} \nonumber
\end{align}  
Using $\Lambda$ and $\Omega$ given in \eqref{spartsex}, we have
$x_1=7\alpha-1$, $x_2=5\alpha-3$, $x_3=4\alpha-4$, $y_1=6\alpha-2$,
$y_2=3\alpha-5$ and $y_3=\alpha-7$.  Substituting these values in
the weights of the 7 alternating sign matrices of size 3 gives
\begin{equation}
\sum_{\mathcal A} w_\alpha(\mathcal A)= -(416\alpha^6+2000\alpha^5+3484\alpha^4+2608\alpha^3+559\alpha^2-256\alpha-108)\alpha^{-6}
\end{equation}
which is, up to a coefficient $-1/\alpha^6$, the non-linear factor in \eqref{ugly} (the rest of the coefficients in the right-hand side of \eqref{detasm} are linear factors).
In fact, before realizing that such a non-linear factor was 
essentially a determinant, we observed that it could naturally be expressed as a weighted sum over
alternating sign-matrices.  This is what led us to suspect a relation with the 6-vertex model, from which we deduced the determinantal form
of the non-linear factors.  In retrospect, 
we ended up with a somewhat unexpected connection between Pieri rules for
Jack and Macdonald polynomials in superspace and alternating sign matrices.
We hope that this will translate into new insight into the combinatorics of alternating sign matrices.

\end{proof}

\begin{acknow}  The authors would like to thank C. Gonz\'alez for useful discussions.
This work was supported by FONDECYT (Fondo Nacional de Desarrollo Cient\'{\i}fico y Tecnol\'ogico de Chile) {postdoctoral grant \#{3160595}} (J. G.) and regular grant \#{1170924} (L. L.).
 \end{acknow}  

\begin{appendix} 
  \section{Dual Pieri rules}  \label{SecPieridual}

  In this section we derive the dual Pieri rules involving an $\alpha$-deformation of the homogeneous symmetric functions in superspace.  The notation is that
  of Section~\ref{SecPieri}.

Let $\hat \omega_\alpha$ be the homomorphism defined on the power-sums by
\begin{equation}
\hat \omega_\alpha(p_n)= (-1)^{n-1} \alpha\, p_n
\qquad {\rm and} \qquad  \hat \omega_\alpha(\tilde p_\ell)= (-1)^{n} \alpha\, \tilde p_\ell 
\end{equation}  
for  $n=1,2,\dots$ and  $\ell=0,1,2,\dots$.  It is known to induce the following
duality on the Jack polynomials in superspace \cite{DLM3}: 
\begin{equation} \label{dualP}
  \hat \omega_\alpha (P_\Lambda^{(\alpha)}) = ||P_\Lambda^{(\alpha)}||^2 \,
  P_{\Lambda'}^{(1/\alpha)}
\end{equation}
where the norm-squared $||P_\Lambda^{(\alpha)}||^2$ is explicitly given by 
\begin{equation}
 \langle \langle P_{\Lambda}^{(\alpha)}, P_{\Lambda}^{(\alpha)} \rangle \rangle_\alpha =: ||P_\Lambda^{(\alpha)}||^2 = \alpha^m \prod_{s \in \Lambda^*}\frac{h_\Lambda^{(\alpha)}(s)}{h^\Lambda_{(\alpha)}(s)}
\end{equation}
with $m$ the fermionic degree of $\Lambda$.
In the special cases $\Lambda=(\emptyset;n)$ and  $\Lambda=(0;\ell)$, the duality
gives respectively
\begin{equation} \label{eqduale}
\hat \omega_\alpha (P_{(\emptyset;n)}^{(\alpha)}) = ||P_{(\emptyset;n)}^{(\alpha)}||^2 \,
P_{(\emptyset;1^n)}^{(1/\alpha)}= e_n \qquad {\rm and} \qquad
\hat \omega_\alpha (P_{(0;\ell)}^{(\alpha)}) = ||P_{(0;\ell)}^{(\alpha)}||^2 \,
P_{(0;1^\ell)}^{(1/\alpha)}= \tilde e_\ell
\end{equation}  
since $P_{(\emptyset;1^n)}^{(1/\alpha)}= m_{(\emptyset; 1^n)}=e_n$ and
$P_{(0;1^\ell)}^{(1/\alpha)}= m_{(0; 1^\ell)}=\tilde e_\ell$ by the triangularity
of the Jack polynomials in superspace (see \eqref{defjack}).
It is thus natural to define
\begin{equation}
  g_n^{(\alpha)} = \frac{1}{||P_{(\emptyset;n)}^{(\alpha)}||^2} \, P_{(\emptyset;n)}^{(\alpha)}
  \qquad {\rm and} \qquad
 \tilde  g_\ell^{(\alpha)} = \frac{1}{||P_{(0;\ell)}^{(\alpha)}||^2} \, P_{(\emptyset;\ell)}^{(\alpha)} 
\end{equation}
which by \eqref{eqduale} are now such that
\begin{equation} \label{dualg}
  \hat \omega_\alpha (g_n^{(\alpha)})=e_n \qquad {\rm and} \qquad
  \hat \omega_\alpha (\tilde g_\ell^{(\alpha)})=\tilde e_\ell
\end{equation}  
We should note that $g_n^{(\alpha)}$ and $\tilde g_\ell^{(\alpha)}$
are respectively equal to the homogeneous symmetric functions in superspace
$h_n$ and $\tilde h_\ell$ (see \eqref{defhomogeneous}) when $\alpha=1$.
As such, they can be considered $\alpha$-deformations of the homogeneous symmetric functions in superspace.  They
turn out to have simple expansions in terms of power sums \cite{DLM2}:
\begin{equation}
  g_n^{(\alpha)} = \sum_{\Lambda \vdash (n|0)} \frac{1}{\alpha^{\ell(\Lambda)} \, z_{\Lambda^s}} \, p_\Lambda \qquad {\rm and} \qquad
  \tilde g_\ell^{(\alpha)} = \sum_{\Lambda \vdash (\ell|1)} \frac{1}{\alpha^{\ell(\Lambda)} \, z_{\Lambda^s}} \, p_\Lambda 
\end{equation}

We will now use the duality to obtain the dual Pieri rules
\begin{equation} \label{pierig}
g_n^{(\alpha)} \, P_\Lambda^{(\alpha)} = \sum_\Omega u_{\Lambda \Omega}(\alpha) \, P_\Omega^{(\alpha)} 
\qquad {\rm and} \qquad \tilde g_n^{(\alpha)} \, P_\Lambda^{(\alpha)} = \sum_\Omega \tilde u_{\Lambda \Omega}(\alpha) \, P_\Omega^{(\alpha)} 
\end{equation}
Apart from a sign, the contributions will be again of two types: a factorized part $\varphi_{\Omega/\Lambda}$ and a determinant ${\rm Det}'_{\Omega/\Lambda}$.
We first describe the factorized part $\varphi_{\Omega/\Lambda}$.  Recall that
${\rm col}_{\Omega/\Lambda}$ denotes the cells of $\Lambda^{\circledast}$ that belong to a column of $\Omega$ containing a non-preexisting cell.   We have that
\begin{equation}
\varphi_{\Omega/\Lambda}  = \prod_{s \in {\rm col}_{\Omega/\Lambda}} \bar c_{\Omega/\Lambda}(s)
\end{equation}
where the contribution $\bar c_{\Omega/\Lambda}(s)$
of a cell $s$ (whose row contains for illustrative purposes all the possible types of cells) is equal to $A$, $B$, or $C$ according to how its column ends:
\begin{equation} \label{fourcases}
  {    \footnotesize\tableau[scY]{ A  &   &\bl $\ldots$  & & \tf \ocircle & \fl & \bl \cerclep \\  \\ \bl \vspace{-2ex}\vdots \\ \bl & \bl \\  & \bl\\  \bl \cerclep     }
    \qquad \qquad
 \footnotesize\tableau[scY]{ C  &   &\bl $\ldots$  & & \tf \ocircle & \fl & \bl \cerclep \\  \\ \bl \vspace{-2ex}\vdots \\ \bl & \bl \\  & \bl\\  \fl  }
 \qquad \qquad
 \footnotesize\tableau[scY]{ B  &   &\bl $\ldots$  & & \tf \ocircle & \fl & \bl \cerclep \\  \\ \bl \vspace{-2ex}\vdots \\ \bl & \bl \\  & \bl\\
   \tf \ocircle  }
 \qquad \qquad
  \footnotesize\tableau[scY]{ C  &   &\bl $\ldots$  & & \tf \ocircle & \fl & \bl \cerclep \\  \\ \bl \vspace{-2ex}\vdots \\ \bl & \bl \\  & \bl\\
   \tf \ocircle \\ \bl \cerclep }
  }
\end{equation}
where $A$, $B$ and $C$ are such as defined in Section~\ref{SecPieri}.
When comparing to  \eqref{sixcases}, the conjugated versions of the two
first cases seem to be missing.  This is due to the fact that the contribution in those cases is now equal to 1 since the corresponding
columns do not contain a non-preexisting cell.

As for the  determinant ${\rm Det'_{\Omega/\Lambda}}$, it is given respectively by
\begin{equation}
  {\rm Det'_{\Omega/\Lambda}} = {\rm Det_{\Omega'/\Lambda'}(1/\alpha)} \qquad {\rm or~by}
  \qquad  {\rm Det'_{\Omega/\Lambda}} = \alpha \, {\rm Det_{\Omega'/\Lambda'}(1/\alpha)}
\end{equation}
whenever $\Omega$ has an extra circle.
The labels $x_i$ and $y_j$ in ${\rm Det}'_{\Omega/\Lambda}$ are thus
naturally taken to
increase from bottom to top instead of from top to bottom while
the determinant has entries
\begin{equation} \label{defxydual}
[x_i;y_j]_{1/\alpha} =
\left \{
\begin{array}{cl}
\dfrac{\alpha}{(y_{i}-x_{j}+1)(y_{i}-x_{j})} & \text{~if there is no~} y  \text{~label in the column of}~x_i \\
1 & \text{~if~} x_i \text{~and~} y_j \text{~are in the same column} \\
0 & \text{~if there is a label~} y_\ell \neq y_j  \text{~in the column of~} x_i 
\end{array} \right .
\end{equation} 
We then have explicitly that 
\begin{equation}
{\rm Det'_{\Omega/\Lambda}} =
\left|
\begin{array}{cccc}
\displaystyle{[x_1; y_1]_{1/\alpha}} & \displaystyle{[x_1; y_2]_{1/\alpha}}  &\cdots & \displaystyle{[x_1; y_n]_{1/\alpha}}  \\ 
\displaystyle{[x_2; y_1]_{1/\alpha}}  & \displaystyle{[x_2; y_2]_{1/\alpha}}  & \cdots & \displaystyle{[x_2; y_n]_{1/\alpha}}  \\
\vdots& \vdots& \ddots& \vdots \\
\displaystyle{[x_n; y_1]_{1/\alpha}}  & \displaystyle{[x_n; y_2]_{1/\alpha}}  & \cdots & \displaystyle{[x_n; y_n]_{1/\alpha}}  
\end{array}
\right| 
\end{equation}
or
\begin{equation}
{\rm Det'_{\Omega/\Lambda}} =
\left|
\begin{array}{cccc}
\alpha & \alpha & \cdots & \alpha \\
\displaystyle{[x_1; y_1]_{1/\alpha}} & \displaystyle{[x_1; y_2]_{1/\alpha}}  &\cdots & \displaystyle{[x_1; y_{n+1}]_{1/\alpha}}  \\ 
\displaystyle{[x_2; y_1]_{1/\alpha}}  & \displaystyle{[x_2; y_2]_{1/\alpha}}  & \cdots & \displaystyle{[x_2; y_{n+1}]_{1/\alpha}}  \\
\vdots& \vdots& \ddots& \vdots \\
\displaystyle{[x_n; y_1]_{1/\alpha}}  & \displaystyle{[x_n; y_2]_{1/\alpha}}  & \cdots & \displaystyle{[x_n; y_{n+1}]_{1/\alpha}}  
\end{array}
\right| 
\end{equation}
when $\Omega$ has an extra circle.
We now have all the tools to state the dual Pieri rules.
\begin{theorem} \label{theodual} The dual Pieri rules for the Jack polynomials in superspace are given by
  \begin{equation} 
g_n^{(\alpha)} \, P_\Lambda^{(\alpha)} = \sum_\Omega u_{\Lambda \Omega}(\alpha) \, P_\Omega^{(\alpha)} 
\qquad {\rm and} \qquad \tilde g_n^{(\alpha)} \, P_\Lambda^{(\alpha)} = \sum_\Omega \tilde u_{\Lambda \Omega}(\alpha) \, P_\Omega^{(\alpha)} 
  \end{equation}
  where the sum is over all $\Omega$'s such that $\Omega/\Lambda$ is  a horizontal $n$-strip and
  a horizontal $\tilde n$-strip  respectively.  Moreover,
\begin{equation} \label{eqPieriexu}
u_{\Lambda \Omega}(\alpha)= (-1)^{\#(\Omega/\Lambda)'}
\varphi_{\Omega/\Lambda} {\rm Det'_{\Omega/\Lambda}}  \qquad {\rm and} \qquad
 \tilde v_{\Lambda \Omega}(\alpha)=  (-1)^{\#(\Omega/\Lambda)'} \varphi_{\Omega/\Lambda} {\rm Det'_{\Omega/\Lambda}}
\end{equation}
where the quantity ${\#(\Omega/\Lambda)'}$ 
stands for the sum 
of the number of preexisting circles and new squares  that lie below each new circle and each bumping square.
\end{theorem}
\begin{proof}
  Using the duality \eqref{dualP} and \eqref{dualg} together with the
  Pieri rules
\eqref{eqPieriTheo}, we have immediately that
  \begin{equation}
    u_{\Lambda \Omega}(\alpha) = v_{\Lambda' \Omega'}(1/\alpha) \, \frac{|| P_{\Omega}^{(\alpha)}||^2}{|| P_{\Lambda}^{(\alpha)}||^2}\qquad {\rm and } \qquad
     \tilde u_{\Lambda \Omega}(\alpha) = \tilde v_{\Lambda' \Omega'}(1/\alpha) \, \frac{|| P_{\Omega}^{(\alpha)}||^2}{|| P_{\Lambda}^{(\alpha)}||^2}
  \end{equation}
It is easy to check that
  \begin{equation}
\frac{|| P_{\Omega}^{(\alpha)}||^2}{|| P_{\Lambda}^{(\alpha)}||^2}= \prod_{s \in \Omega^*} A(s) B(s)  \qquad  {\rm or }  \qquad   \frac{|| P_{\Omega}^{(\alpha)}||^2}{|| P_{\Lambda}^{(\alpha)}||^2}= \alpha \prod_{s \in \Omega^*} A(s) B(s)
  \end{equation}
depending on whether $\Omega$ has an extra circle (note that
  the dependency of $A$ and $B$ on the cell $s$ has been emphasized).
Using  
\begin{equation}
A'(1/\alpha)= \frac{1}{B} \qquad {\rm and }
\qquad B'(1/\alpha)= \frac{1}{A}
\end{equation}
where the prime indicates that we are considering the conjugate of the superpartitions, we have immediately that
\begin{equation}
\varphi_{\Omega/\Lambda}(\alpha) = \psi_{\Omega'/\Lambda'}'(1/\alpha)  \frac{|| P_{\Omega}^{(\alpha)}||^2}{|| P_{\Lambda}^{(\alpha)}||^2} \qquad {\rm or }
\qquad \alpha\,  \varphi_{\Omega/\Lambda}(\alpha) = \psi_{\Omega'/\Lambda'}'(1/\alpha)  \frac{|| P_{\Omega}^{(\alpha)}||^2}{|| P_{\Lambda}^{(\alpha)}||^2}
\end{equation}
depending as usual on whether $\Omega$ contains an additional circle.  The theorem then follows from the definition of ${\rm Det'_{\Omega/\Lambda}}$ (which takes care of the possible additional $\alpha$) and from the fact that
\begin{equation}
   {\#(\Omega/\Lambda)'} =  {\#(\Omega'/\Lambda')}
\end{equation}  
\end{proof}  

\end{appendix}

\end{document}